\documentclass[12pt,a4paper]{article}
\usepackage{amssymb}

\usepackage{amssymb}
\usepackage{hyperref}

\usepackage{graphicx,amssymb,amsfonts,epsfig,a4,amsmath,amsthm,xypic}
\usepackage[latin1]{inputenc}

\usepackage{tikz-cd}
\usepackage{comment}

\newtheorem{Thm}{Theorem}[section]
\newtheorem{Prop}[Thm]{Proposition}
\newtheorem{Lem}[Thm]{Lemma}
\newtheorem{Cor}[Thm]{Corollary}

\newtheorem{Rem}[Thm]{Remark}

\newtheorem{Conj}{Conjecture}

\newcommand{\Z}{\mathbb{Z}}

\newcommand{\R}{\mathbb{R}}
\newcommand{\C}{\mathbf{C}}

\newcommand{\Ker}{\text{Ker}}

\newcommand{\Id}{\text{Id}}

\newcommand{\Cs}{\mathrm{C}^*}

\newcommand{\ra}{\rightarrow}

\title{Equivariant K-homology and K-theory for some discrete planar affine groups} 
\author{Ramon FLORES, Sanaz POOYA, \\
and Alain VALETTE}
\date{\today}

\begin{document}

\maketitle

\begin{abstract} We consider the semi-direct products $G=\Z^2\rtimes GL_2(\Z), \Z^2\rtimes SL_2(\Z)$ and $\Z^2\rtimes\Gamma(2)$ (where $\Gamma(2)$ is the congruence subgroup of level 2). For each of them, we compute both sides of the Baum-Connes conjecture, namely the equivariant $K$-homology of the classifying space $\underline{E}G$ for proper actions on the left-hand side, and the analytical K-theory of the reduced group $C^*$-algebra on the right-hand side. The computation of the LHS is made possible by the existence of a 3-dimensional model for $\underline{E}G$, which allows to replace equivariant K-homology by Bredon homology. We pay due attention to the presence of torsion in $G$, leading to an extensive study of the wallpaper groups associated with finite subgroups. For the first and third groups, the computations in $K_0$ provide explicit generators that are matched by the Baum-Connes assembly map. 
\end{abstract}

\section{Introduction}

Let $G$ be a countable discrete group. The {\it Baum-Connes conjecture for $G$} posits an isomorphism:
\begin{Conj} (BC) The map
$$\mu_i^G: K_i^G(\underline{E}G)\rightarrow K_i(C^*_r(G))\qquad i=0,1$$
is an isomorphism.
\end{Conj}

Here $\mu_G$ is the Baum-Connes assembly map, $K_i^G(\underline{E}G)$ denotes the $G$-equivariant K-homology of the classifying space $\underline{E}G$ of $G$-proper actions, and $K_i(C^*_r(G))$ denotes the topological K-theory of the reduced $C^*$-algebra of $G$: see \cite{BCH, GoJuVa} for background.

The Baum-Connes conjecture, and even the more general Baum-Connes conjecture with coefficients, have been established for large classes of groups: e.g. groups with the Haagerup property (a deep result by Higson-Kasparov \cite{HiKa}), or word-hyperbolic groups (another deep result by V. Lafforgue \cite{Laff}). Those results are proved by powerful methods which, in their generality, do not allow us to compute explicitly the abelian groups appearing in the conjecture.

In this paper we deal with some semi-direct products $\Z^2\rtimes H$ where $H$ is a finite-index subgroup of $GL_2(\Z)$. More precisely we study the following three groups:
\begin{enumerate}
\item $SA=:\Z^2\rtimes SL_2(\Z)$, the {\it special affine group} of $\Z^2$.
\item $GA=:\Z^2\rtimes GL_2(\Z)$, the {\it general affine group} of $\Z^2$.
\item $\Z^2\rtimes\Gamma(2)$,  where $\Gamma(2)$ is the {\it principal congruence subgroup of level 2} in $SL_2(\Z)$, i.e. the kernel of the reduction modulo 2:
$$\Gamma(2)=\ker(SL_2(\Z)\rightarrow SL_2(\Z/2\Z)).$$
It is a normal subgroup of index 6 in $SL_2(\Z)$. 
\end{enumerate}

We immediately point out that the Baum-Connes conjecture, in its more general version with coefficients, is known to hold for $\Z^2\rtimes H$, with $H\subset GL_2(\Z)$. Indeed the Haagerup property implies $\Z^2\rtimes H$ acts on a tree with amenable stabilizers, see also Proposition \ref{treeamenstab} below. So, we may appeal to a result of H. Oyono-Oyono \cite{Oyono}, who proved that the Baum-Connes conjecture with coefficients is stable under group actions on trees. More precisely: let $G$ be a group acting without inversion on a tree; then $G$ satisfies the Baum-Connes conjecture with coefficients if and only if vertex-stabilizers in $G$ do satisfy it.

\medskip
We were attracted to the three semi-direct products above basically for two reasons: 
\begin{itemize}
    \item They do not belong to the two large classes mentioned above: they are not hyperbolic (they contain $\Z^2$), and they do not have the Haagerup property (indeed the pair $(\Z^2\rtimes H,\Z^2)$ has the relative property (T), see \cite{Burger}, p. 62). 
    \item They contain torsion elements; this makes the computations more involved, as we have to work with equivariant K-homology of the classifying space for proper actions $\underline{E} G$, instead of the simpler ordinary K-homology of the usual classifying space $BG$, as would be the case for torsion-free group $G$.
    \footnote{The case of $\Z^2\rtimes H$ with $H$ torsion-free is of course also interesting, see a paper in preparation by the third author with Alexandre Zumbrunnen, which deals with the case when $H$ is a free group of rank 2. The results are especially complete when $H=S$, the Sanov subgroup of $SL_2(\Z)$, or $H=C$, the commutator subgroup of $SL_2(\Z)$.}. Therefore we use Bredon homology to compute equivariant K-homology (see Section \ref{BredonSA}).
\end{itemize}

For the three groups above, we compute explicitly both sides of the Baum-Connes conjecture \footnote{For $SA$, the K-theory of the reduced $C^*$-algebra had already been computed by S. Walters \cite{Walters}; we offer a different computation in Section \ref{RHS-AS}.}, and prove that they are abstractly isomorphic. Moreover, for $SA$ and $\Z^2\rtimes \Gamma(2)$ we check by hand that the assembly map $\mu_0$ is an isomorphism. We summarize our results in the following table: 

\bigskip
\begin{center}
\begin{tabular}{c|c|c|c}
  Group   & $K_0$ & $K_1$ & Results \\
  \hline
   $SA$  & $\Z^{14}$ & $\Z$ & \ref{LHSforSA}, \ref{RHSforSA}  \\
   $GA$ &  $\Z^{11}$ & 0 & \ref{LHSforGA}, \ref{RHSforGA}  \\
   $\Z^2\rtimes\Gamma(2)$ & $\Z^6$ & $\Z^{12}$ & \ref{LHSforZ2rtimesGamma(2)}, \ref{RHSforZ2rtimesGamma(2)} 
\end{tabular}
\end{center}

Let us describe briefly the methods we use in this paper. 
\begin{itemize}
    \item The existence of a 3-dimensional model for $\underline{E}G$ (see Corollary \ref{3Dmodel} below) allows for the computation of the LHS. So we may appeal to a result by G. Mislin (see Theorem \ref{Mislin} below) that reduces the computation of $K_i^G(\underline{E}G)$ to the computation of the Bredon homology $H_j^{\mathfrak{F}}(\underline{E}G,R_{\mathbb{C}})$ (with $0\leq j\leq 3$), as explained in Section \ref{BredonSA}.
    \item We observe that the groups in consideration act naturally on a tree with edge and vertex stabilizers being planar wallpaper groups: six out of the seventeen wallpaper groups do occur, namely the groups $\mathbf{p2}$, $\mathbf{p4}$, $\mathbf{p6}$, $ \mathbf{cmm}$, $\mathbf{p4m}$  and $\mathbf{p6m}$. Our Proposition \ref{FromOyono} allows to reduce computations of both sides of BC for our groups, to the case of the corresponding groups for wallpaper groups. 
    \item For the LHS of the assembly map for wallpaper groups, we take advantage of the extensive computations in \cite{Flores}. 
    
    \item For the RHS of the assembly map for wallpaper groups, we appeal to \cite{ELPW} for the K-theory of the $C^*$-algebra of $\mathbf{p2},\mathbf{p4},\mathbf{p6}$. But for $\mathbf{cmm}$, $\mathbf{p4m}$ and $\mathbf{p6m}$ we must perform the computations, that appear in Sections \ref{cmm}, \ref{p4m}, \ref{p6m} respectively. The results are then plugged into Pimsner's 6-terms exact sequence \cite{Pimsner} that relates K-theory of the reduced $C^*$-algebra of $G$ to K-theory of the reduced $C^*$-algebras of vertex- and edge-stabilizers. Observe that, since our groups act on a tree with amenable stabilizers, it follows from \cite{JuVa} that they are K-amenable. This means that, throughout the paper, we may replace reduced group $C^*$-algebras by maximal (or full) group $C^*$-algebras.

\end{itemize}

\medskip

The limit of our methods are displayed e.g. by the case of $SA$: the $K_1$-groups are infinite cyclic, but we do not have any good description of a generator, neither on the LHS nor on the RHS \footnote{However we have an explicit description of the images of the $K_1$-groups under the connecting maps in either the Mayer-Vietoris or the Pimsner exact sequences, and the connecting maps are isomorphisms.}.

The paper is structured as follows. Section 2 contains preliminaries on BC for groups acting on trees (see especially Proposition \ref{FromOyono}), on models of $\underline{E}H$ when $H$ is a discrete subgroup of the affine group of $\R^n$ (see especially Corollary \ref{modelforEunderbar}), on $SL_2(\Z)$ and $GL_2(\Z)$ viewed as amalgams, and on Bredon homology. Then we present our computations for $SA$ in Section 3, for $GA$ in Section 4, and for $\Z^2\rtimes \Gamma(2)$ in Section 5. We chose to present computations in that order because, for $SA$, the wallpaper groups involved, namely $\mathbf{p2},\mathbf{p4},\mathbf{p6}$, are extensions of $\Z^2$ by finite cyclic groups, so they are easier to handle.

\medskip
ACKNOWLEDGEMENTS: We thank Sven Raum and Alexandre Zumbrunnen for various valuable discussions. We also acknowledge helpful comments by Yago 
Antol\'{i}n, Mathieu Dutour Sikiri\'{c} and Ian Leary. We warmly thank Doris Schattschneider for the permission to reproduce the picture (Figure \ref{Doris}) of the patterns of the wallpaper groups that appear in the nice paper \cite{Sch78}. We are very grateful to the referee for an extremely careful reading of the first version, and detailed suggestions which greatly improved the exposition.

R.F. was supported by grants PID2020-117971GB-C21 of the Ministery of Science of Innovation of Spain and US-1263032 of the Junta de Andalucia. S.P. was supported by grant KAW 2020.0252 of the Knut and Alice Wallenberg Foundation.


\section{Preliminaries}

\subsection{BC and groups acting on trees}
In this subsection we extract  from Oyono's paper \cite{Oyono} the fact that, not only the Baum-Connes conjecture with coefficients, but also the ordinary Baum-Connes conjecture, is inherited by group actions on trees. This must be contrasted with the situation for direct products of two groups: it is known that the Baum-Connes conjecture with coefficients is inherited by direct products, but it is an open question whether the ordinary Baum-Connes conjecture is so.

\begin{Prop}\label{FromOyono} Let $G$ be a group acting without inversion on a tree. Assume that edge- and vertex-stabilizers satisfy $BC$, then $G$ satisfies $BC$.
\end{Prop}

{\bf Proof:} The proof can be extracted from the results in \cite{Oyono}. First, let $(G_i)_{i\in I}$ be a family of subgroups of $G$ that satisfy $BC$. Let $G$ act on a discrete space $Z$ in such a way that any stabilizer from $Z$ is conjugate to some $G_i$. Then by Theorem 3.2 in \cite{Oyono} (with $A=\mathbb C$), $G$ satisfies $BC$ with coefficients in $C_0(Z)$.

Now suppose that $G$ acts on a tree with vertex set $X_0$ and edge set $X_1$. By Proposition 4.5 in \cite{Oyono}, there is a 6-terms exact sequence

\begin{equation*}
    \xymatrix{
    KK_0^G(\underline{E}G,C_0(X_1)) \ar[r] &
    KK_0^G(\underline{E}G, C_0(X_0)) \ar[r] &
    K_0^G(\underline{E}G)\ar[d] \\
    K_1^G(\underline{E}G)\ar[u] &
    KK_1^G(\underline{E}G,C_0(X_0)) \ar[l] &
    KK_1^G(\underline{E}G,C_0(X_1)) \ar[l]
    }
\end{equation*}
such that the assembly maps intertwine this sequence and Pimsner's exact sequence \cite{Pimsner}:
\begin{equation*}
    \xymatrix{
    K_0(C_0(X_1)\rtimes_r G) \ar[r]&
    K_0(C_0(X_0)\rtimes_r G) \ar[r]&
    K_0(C^*_r(G)) \ar[d]\\
    K_1(C^*_r(G)) \ar[u] &
    K_1(C_0(X_0)\rtimes_r G) \ar[l]&
    K_1(C_0(X_1)\rtimes_r G) \ar[l]
    }
\end{equation*}

Assume that vertex- and edge-stabilizers satisfy $BC$, then $G$ satisfies $BC$ with coefficients in $C_0(X_0)$ and in $C_0(X_1)$, so the result follows from the five lemma.
\hfill$\square$

\subsection{The classifying space for proper actions for certain affine groups}\label{ExplicitClassifying}

Let $V$ be a finite-dimensional real vector space, and let $Aff(V)=V\rtimes GL(V)$ be its affine group. Let $\ell:Aff(V)\rightarrow GL(V):g\mapsto \ell(g)$ be the quotient homomorphism, so that $\ell(g)$ is the linear part of $g$. We write $b(g)$ for the translation part of $g$, so that $g v=\ell(g)v+b(g)$ for every $v\in V$.

Let $G$ be a countable subgroup of $Aff(V)$, $X$ a model for $\underline{E}\ell(G)$. We view $X$ as a $G$-space via the quotient map $\ell$. We endow $V\times X$ with the diagonal action.

\begin{Prop} Assume that $b(G)$ is closed and discrete in $V$. Then $V\times X$ is a model for $\underline{E}G$.
\end{Prop}

{\bf Proof:} By definition/characterization of a model for $\underline E G$ (see for example \cite[Definition 0.1]{Lueck00}, we need to show that  $G$ acts properly on such a CW-complex and its fixed point sets are contractible when taken by finite subgroups of $G$ and trivial otherwise.
Let us first show properness of the action $G\curvearrowright V\times X$. Let $K\subset V, L\subset X$ be compact subsets. We must show that
$$\{g\in G \mid g(K\times L)\cap(K\times L)\neq\emptyset\}$$
is finite. This set equals
$$\left\{g\in G \mid gK\cap K\neq\emptyset\right\}
\cap
\left\{g\in G \mid \ell(g)L\cap L\neq\emptyset\right\}.$$
By properness of the action of $\ell(G)$ on $X$, finitely many $\ell(g)'s$ appear in the latter set. Assume $\ell(g)$ is fixed. Then $gK\cap K=(b(g)+\ell(g)K)\cap K$. Since $V$ acts properly on itself, the set $F=:\{v\in V \mid (v+\ell(g)K)\cap K\neq\emptyset\}$ is compact in $V$. By assumption $b(G)$ is closed and discrete, and therefore $b(G)\cap F$ is finite. So for given $\ell(g)$, we have finitely many possible $b(g)$'s.

Now let us show the statement concerning fixed point sets. Let $H$ be a subgroup of $G$. We have $(V\times X)^H=V^H\times X^{\ell(H)}$. Observe that $V^H$ is a (possibly empty) affine subspace of $V$.
\begin{itemize}
\item If $H$ is finite, then $V^H$ is non-empty (for every $v\in V$, the point $\frac{1}{|H|}\sum_{h\in H}h v$ is in $V^H$), so that $V^H$ is contractible. Since $X^{\ell(H)}$ is contractible, so is $(V\times X)^H$.
\item If $H$ is infinite, two cases may occur. If $\ell(H)$ is infinite, then $X^{\ell(H)}=\emptyset$. If $\ell(H)$ is finite, then $V\cap H$ is infinite, and $V^{V\cap H }=\emptyset$.
\end{itemize}
\hfill$\square$

As an immediate consequence we obtain:

\begin{Cor}\label{modelforEunderbar} Assume that $G$ is a semi-direct product $G=A\rtimes B$ with $A$ a discrete subgroup of $V$ and $B\subset GL(V)$ that leaves $A$ invariant, i.e. $B(A)=A$. Then a model for $\underline{E}G$ is given by $V\times\underline{E}B$.
\hfill$\square$
\end{Cor}

\subsection{\texorpdfstring{$SL_2(\Z)$}{SL\_2(Z)} and \texorpdfstring{$GL_2(\Z)$}{GL\_2(Z)} as amalgamated products}\label{GL2amalgam}

We denote by $C_n$ the cyclic group of order $n$, and by $D_n$ the dihedral group of order $2n$ (so that $D_2\simeq C_2\times C_2$ is the \emph{Vierergruppe}). Set 
$$
S=\left(\begin{array}{cc}0 & -1 \\1 & 0\end{array}\right), \quad 
T=\left(\begin{array}{cc}1 & 1 \\0 & 1\end{array}\right), \quad
U=\left(\begin{array}{cc}1 & 0 \\1 & 1\end{array}\right), \quad
\varepsilon = 
\begin{pmatrix}
   -1 & 0 \\
   0 & -1 
\end{pmatrix}
,$$
and  $R=ST=\left(\begin{array}{cc}0 & -1 \\1 & 1\end{array}\right)$. It is well-known that $SL_2(\Z)$ is generated by $S$ and $T$. Moreover $S$ has order $4$, and $R$ has order 6, with $S^2=R^3=\varepsilon$: as in Example (c) on p. 52 of \cite{Serre}, this exhibits $SL_2(\Z)$ as the amalgamated product:
$$SL_2(\Z)\simeq C_4\ast_{C_2} C_6.$$

Since every matrix in $GL_2(\Z)\setminus SL_2(\Z)$ is uniquely the product of $\left(\begin{array}{cc}0 & 1 \\1 & 0\end{array}\right)$ with an element of $SL_2(\Z)$, and this matrix normalizes the factors in the amalgam decomposition of $SL_2(\Z)$, we also get an amalgamated product decomposition of $GL_2(\Z)$ (see Section 6 in \cite{BuTa}):
$$GL_2(\Z)\simeq D_4\ast_{D_2} D_6.$$

\begin{Prop}\label{treeamenstab} $\Z^2\rtimes GL_2(\Z)$ (and hence any of its subgroups) acts on a tree with stabilizers which are planar wallpaper groups. In particular the Baum-Connes conjecture BC holds for $\Z^2\rtimes GL_2(\Z)$ and any of its subgroups.
\end{Prop}

{\bf Proof:} Let $T$ be the tree associated to the previous amalgamated decomposition of $GL_2(\Z)$ via Bass-Serre theory. Hence $GL_2(\Z)$ acts on $T$ with finite stabilizers. Through the quotient map $\Z^2\rtimes GL_2(\Z)\rightarrow GL_2(\Z)$, any subgroup of $\Z^2\rtimes GL_2(\Z)$ acts on $T$ with stabilizers of the form $\Z^2\rtimes F$, with $F$ a finite subgroup of $GL_2(\Z)$, and those semi-direct products are planar wallpaper groups. As Baum-Connes holds for these subgroups of $\Z^2\rtimes GL_2(\Z)$, the last statement follows from Proposition \ref{FromOyono}.
\hfill$\square$

\medskip
As it was mentioned in the Introduction:

\begin{Cor}\label{3Dmodel} Let $B$ be any subgroup of $GL_2(\Z)$. The semi-direct product $G=\Z^2\rtimes B$ admits a 3-dimensional model for $\underline{E}G$. If $B$ is not a torsion subgroup, there is no such model of dimension $<3$. 
\end{Cor}

{\bf Proof:} We apply Corollary \ref{modelforEunderbar} with $A=\Z^2$. A model for $\underline{E}G$ is $\R^2\times\underline{E}B$, and we may take for $\underline{E}B$ the Bass-Serre tree $T$ from the proof of Proposition \ref{treeamenstab}.

To prove the second statement: as $B$ is not torsion, it contains subgroups isomorphic to $\Z$, and then the pullback of any of these subgroups by our extension gives subgroups $H\subset G$ that are polycyclic extensions $\Z^2\rightarrow H\rightarrow \Z$. Recall from \cite[Section 5.4.13]{Rob82} that the \emph{Hirsch length} or \emph{Hirsch rank} of a polycyclic group is the number of infinite factors any polycyclic series of the group. This number is a lower bound for the dimension of any model of the classifying space for proper actions of the group (see \cite[Example 5.26]{Lueck05}). Hence, as the Hirsch length of $H$ is 3, the minimal dimension of any model of $\underline{E}G$ is exactly 3.




\subsection{The left-hand side of BC: equivariant K-homology}\label{BredonSA}


We denote by  $\mathfrak{F}$ the family of finite subgroups of a given group. We refer to the excellent survey of Mislin in \cite{MV03} for more background concerning proper actions, classifying spaces for various families, Bredon homology and equivariant $K$-homology. Let us recall the results that will be crucial for our computations. The first concerns existence of a Mayer-Vietoris sequence in Bredon homology.

\begin{Thm}[\cite{MV03}, Corollary 3.32]

Let $\Gamma=H\ast_LK$, with $H$, $L$ and $K$ be arbitrary groups, and let $N\in \Gamma\textrm{-}\textrm{Mod}_{\mathfrak{F}}$. Then there exists a long exact sequence

$$
\ldots \rightarrow H_n^{\mathfrak{F}}(\underline{E} L,N) \rightarrow H_n^{\mathfrak{F}}(\underline{E} H,N)\oplus H_n^{\mathfrak{F}}(\underline{E} K,N)\rightarrow H_n^{\mathfrak{F}}(\underline{E}\Gamma,N) \rightarrow H_{n-1}^{\mathfrak{F}}(\underline{E} L,N)\rightarrow\ldots 
$$

\label{Vietoris}
\end{Thm}

The maps of the previous sequences are described in the proof of Lemma 3.31 of \cite{MV03}, where it is also shown that the homomorphisms are induced by group inclusions. Thorough analysis of the functoriality of the Bredon (co)homology can be found in \cite[Chapter 2]{Flu11} or \cite[Chapter 11]{Lu23}.

 In order to describe  $K_*^ {\Gamma} (\underline{E}\Gamma)$ we need to compute the Bredon homology with coefficients in the complex representation ring functor $R_{\mathbb{C}}$, and then to make use of  Davis-L\"{u}ck \cite{DaLu98} version of the Atiyah-Hirzebruch spectral sequence to switch to equivariant $K$-homology. Usually, a manageable model for $\underline{E}\Gamma$ makes the computations more feasible. 
 
 \begin{Thm}[\cite{MV03}, Theorem 5.29]
\label{Mislin}
Let $\Gamma$ be a group such that there exists a model for $\underline{E}\Gamma$ of dimension 3. Then there is a natural exact sequence

\begin{equation*}
\xymatrix{ 0 \ar[r] & H_1^{\mathfrak{F}}(\underline{E}\Gamma,R_{\mathbb{C}}) \ar[r] & K_1^{\Gamma}(\underline{E}\Gamma) \ar[r] & H_3^{\mathfrak{F}}(\underline{E}\Gamma,R_{\mathbb{C}}) \ar[d] \\ 
0  & H_2^{\mathfrak{F}}(\underline{E}\Gamma,R_{\mathbb{C}}) \ar[l] & K_0^{\Gamma}(\underline{E}\Gamma) \ar[l] &
H_0^{\mathfrak{F}}(\underline{E}\Gamma,R_{\mathbb{C}}) \ar[l]
}
\end{equation*}

\label{BredonK}
\end{Thm}

Note that by Corollary \ref{3Dmodel} there is a 3-dimensional model for the groups under consideration in this paper. Therefore the above theorem will be our main tool to describe the K-homology groups. 

\section {K-homology and K-theory for \texorpdfstring{$SA=\Z^2\rtimes SL_2(\Z)$}{SA=Z\textasciicircum x SL\_2(Z)}}

We denote by $SA$ the semi-direct product $\Z^2 \rtimes SL_2(\Z)$, where every element of the free abelian group is identified with a vertical vector $\binom{x}{y}$ and the (left) action is by product of matrices. 
\subsection{Bredon homology and equivariant K-homology of \texorpdfstring{$SA$}{SA}}

In this section we compute the Bredon homology of $SA$, and as a consequence the equivariant $K$-homology of $\underline{E}SA$. The main reference for this section is \cite{Flores}, where Bredon homology of wallpaper groups is computed. 

It was already recalled in Subsection \ref{GL2amalgam} that $SL_2(\Z)$ is isomorphic to the amalgamated free product  $SL_2(\Z)=C_4\ast_{C_2} C_6$, with
$C_4 = \langle S\rangle$ and $C_6 = \langle R \rangle$ over $C_2 = \langle \varepsilon\rangle$ via the identification $S^2=R^3=\varepsilon$, where
\[
\varepsilon = 
\begin{pmatrix}
   -1 & 0 \\
   0 & -1 
\end{pmatrix}
,\quad
S=
\begin{pmatrix}
   0 & -1 \\
   1 & 0  
\end{pmatrix}
,\quad
R=
\begin{pmatrix}
  0 & -1 \\
   1 & 1  
\end{pmatrix}.
\]
Moreover, the pullback of the diagram $ \langle \varepsilon \rangle \hookrightarrow SL_2(\Z) \leftarrow  SA$ is isomorphic to $\mathbf{p2}$, the pullback of $ \langle S \rangle \hookrightarrow SL_2(\Z) \leftarrow  SA$ is isomorphic to $\mathbf{p4}$, and the pullback of $ \langle R \rangle \hookrightarrow SL_2(\Z) \leftarrow  SA$ is isomorphic to $\mathbf{p6}$. Now, the inclusions $\mathbf{p4}\hookrightarrow SA$ and $\mathbf{p6}\hookrightarrow SA$ of the pullbacks induce a homomorphism $\mathbf{p4}\ast \mathbf{p6}\rightarrow SA$ which is surjective and factors through another homomorphism $\mathbf{p4}\ast_{\mathbf{p2}} \mathbf{p6}\rightarrow SA$. It is then routine to check that the latter is an isomorphism.

Taking into account that Bredon homology of wallpaper groups is trivial above dimension 2, and the fact that there is a model of $\underline{E}G$ of dimension 3 (Corollary \ref{3Dmodel}), the Mayer-Vietoris sequence of Theorem \ref{Vietoris} reduces to

\begin{equation} \label{MV-SA}
\footnotesize{
\begin{tikzcd}
  &  
  \hskip 13em 0
  \rar 
  &  
  H_3^{\mathfrak F}(\underline{E}SA,R_{\mathbb{C}}) 
   \ar[draw=none]{d}[name=X, anchor=center]{}
 \ar[rounded corners,
            to path={ -- ([xshift=2ex]\tikztostart.east)
                      |- (X.center) \tikztonodes
                      -| ([xshift=-2ex]\tikztotarget.west)
                      -- (\tikztotarget)}]{dll}[at end,above]{} 
                      \\    
  H_2^{\mathfrak F}(\underline{E}\mathbf{p2},R_{\mathbb{C}})
  \rar
  & 
  H_2^{\mathfrak F}(\underline{E}\mathbf{p4},R_{\mathbb{C}})\oplus H_2^{\mathfrak F}(\underline{E}\mathbf{p6},R_{\mathbb{C}})
  \rar
  &
  H_2^{\mathfrak F}(\underline{E}SA,R_{\mathbb{C}})
             \ar[draw=none]{d}[name=X, anchor=center]{}
 \ar[rounded corners,
            to path={ -- ([xshift=2ex]\tikztostart.east)
                      |- (X.center) \tikztonodes
                      -| ([xshift=-2ex]\tikztotarget.west)
                      -- (\tikztotarget)}]{dll}[at end,above]{} 
    \\                 
  H_1^{\mathfrak F}(\underline{E}\mathbf{p2},R_{\mathbb{C}})
  \rar
  &
  H_1^{\mathfrak F}(\underline{E}\mathbf{p4},R_{\mathbb{C}})
  \oplus 
  H_1^{\mathfrak F}(\underline{E}\mathbf{p6},R_{\mathbb{C}})
  \rar
  &
  H_1^{\mathfrak F}(\underline{E}SA,R_{\mathbb{C}})
  			 \ar[draw=none]{d}[name=X, anchor=center]{}
  \ar[rounded corners,
            to path={ -- ([xshift=2ex]\tikztostart.east)
                      |- (X.center) \tikztonodes
                      -| ([xshift=-2ex]\tikztotarget.west)
                      -- (\tikztotarget)}]{dll}[at end,above]{}
                      \\
 H_0^{\mathfrak F}(\underline{E}\mathbf{p2},R_{\mathbb{C}})
 \rar
 &
 H_0^{\mathfrak F}(\underline{E}\mathbf{p4},R_{\mathbb{C}})
 \oplus 
 H_0^{\mathfrak F}(\underline{E}\mathbf{p6},R_{\mathbb{C}})
 \rar
 &
 H_0^{\mathfrak F}(\underline{E} SA,R_{\mathbb{C}})
\rar
&
0
\end{tikzcd}
} 
\end{equation}

For the sake of completeness, we recall in Table \ref{Bredon homology} the Bredon homology of the three wallpaper groups involved in the structure of $SA$, and include a picture of the fundamental domains of all these groups. We refer to Section 3 in \cite{Flores} for detailed descriptions of $H_*^{\mathfrak{F}}$ and as well as representations and the cell structure of the models for $\mathbf{p2}$, $\mathbf{p4}$ and $\mathbf{p6}$. We remark anyhow the following facts for the benefit of the reader.

\begin{figure}
\label{Doris}
\includegraphics{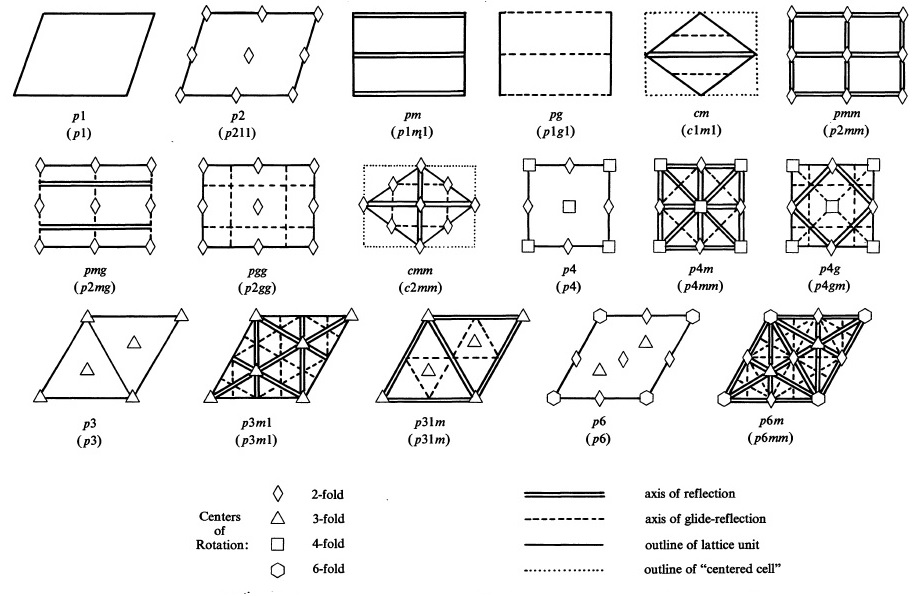}
\caption{Patterns for the wallpaper groups}
\end{figure}

\begin{itemize}

\item For $\mathbf{p2}$, there are three classes $e_i$, $0\leq i\leq 2$ of 0-equivariant cells whose stabilizers are isomorphic to $C_2$. Then $\alpha_i^j$ denotes a generator of the copy of $R_{\C}(C_2)$ corresponding to $e_i$ (recall that the representation ring of the cyclic group $C_n$ is isomorphic to $\Z^n$).

\item For $\mathbf{p4}$, there are two classes $e_0$ and $e_2$ of 0-equivariant cells whose stabilizers are isomorphic to $C_4$, and one class $e_1$ whose stabilizer is isomorphic to $C_2$. Then $\alpha_i^j$ denotes a generator of the copy of $R_{\C}(C_4)$ corresponding to $e_i$, $i=0,2$, and $\alpha_1^j$ denotes a generator of the copy of $R_{\C}(C_2)$ corresponding to $e_1$.

\item For $\mathbf{p6}$, there are three classes $e_0$ and $e_2$ $e_i$, $0\leq i\leq 2$ of 0-equivariant cells whose stabilizers are respectively isomorphic to $C_6$, $C_3$ and $C_2$ Then $\alpha_0^j$ denotes a generator of the copy of $R_{\C}(C_6)$ corresponding to $e_0$, $\alpha_1^j$ denotes a generator of the copy of $R_{\C}(C_3)$ corresponding to $e_1$, and $\alpha_1^j$ denotes a generator of the copy of $R_{\C}(C_2)$ corresponding to $e_2$.

\end{itemize}

In all of the characters with the superscript $0$ stands for the trivial representation.

\begin{table}[t]
\begin{center}
\begin{tabular}{| c | c | c | c | c |}
\hline
Group & $H_2^{\mathfrak{F}}$ & $H_1^{\mathfrak{F}}$ & $H_0^{\mathfrak{F}}$ & Basis $H_0^{\mathfrak{F}}$ \\ \hline

$\mathbf{p2} $  & $ \Z $ & $ 0 $ & $ \Z^5 $ & $ ([\alpha_0^1], [\alpha_0^2], [\alpha_1^2], [\alpha_2^2], [\alpha_3^2]) $\\ \hline
$\mathbf{p4} $  & $ \Z $ & $ 0 $ & $ \Z^8 $ & $ ([\alpha_0^1],[\alpha_0^2],[\alpha_0^3],[\alpha_0^4],[\alpha_1^2],[\alpha_2^2],[\alpha_2^3],[\alpha_2^4]) $\\ \hline
$\mathbf{p6} $ & $ \Z $ & $ 0 $  & $ \Z^9 $ & $ ([\alpha_0^2],[\alpha_0^3],[\alpha_0^4],[\alpha_0^5],[\alpha_0^6],[\alpha_1^2],[\alpha_1^3],[\alpha_2^1],[\alpha_2^2]) $\\ \hline

\end{tabular}
\caption{Bredon homology}
\label{Bredon homology}
\end{center}
\end{table}

The description of the Bredon homology groups of wallpaper groups in Table \ref{Bredon homology} and in particular the vanishing of $H_1^{\mathfrak F}$ cut the long exact sequence \eqref{MV-SA} into two short exact sequences that we need to analyze. We start with the top degree:
\begin{equation} \label{2-degree-SA}
    0\rightarrow H_3^{\mathfrak F}(\underline{E}SA,R_{\mathbb{C}})\rightarrow \Z\rightarrow \Z^2 \rightarrow H_2^{\mathfrak F}(\underline{E}SA,R_{\mathbb{C}}) \rightarrow 0.
\end{equation}

Thanks to Lemma 3.21 in \cite{MV03}, the middle homomorphism $ \Z\rightarrow \Z^2$ represents the Mayer-Vietoris homomorphism in ordinary homology $H_2(\mathbf{p2},\Z)\rightarrow H_2(\mathbf{p4},\Z)\oplus H_2(\mathbf{p6},\Z)$ induced by the inclusions $\mathbf{p2}\hookrightarrow \mathbf{p4}$ and $\mathbf{p2}\hookrightarrow \mathbf{p6}$. 
The first inclusion induces a map $S^2\rightarrow S^2$ (where the spheres in the left and the right are seen as models for the orbit spaces $\underline{B}\mathbf{p2}$ and $\underline{B}\mathbf{p4}$, respectively), which is a branched 2-fold cover of the sphere by itself, hence induces multiplication by 2 between the second homology groups.
Analogously, the inclusion $\underline{B}\mathbf{p2}$ and $\underline{B}\mathbf{p6}$ induces multiplication by 3. 
Then $H_2^{\mathfrak F}(\underline{E}SA,R_{\mathbb{C}})$ appears as the cokernel of the homomorphism $ \Z\rightarrow \Z^2$ that takes the generator of $\Z$ to the element $(2,-3)\in \Z^2$, and is isomorphic to $\Z$. 
In turn, $H_3^{\mathfrak F}(\underline{E} SA,R_{\mathbb{C}})$ is the kernel of an injective homomorphism, hence trivial. The negative sign comes from the definition of Mayer-Vietoris sequence.

We now focus on the bottom row of the Mayer-Vietoris exact sequence \eqref{MV-SA}. 
Substituting information from Table \ref{Bredon homology}, we obtain

\begin{equation} \label{0-degree-SA}
    0\rightarrow H_1^{\mathfrak F}(\underline{E}SA,R_{\mathbb{C}})\rightarrow \Z^5\stackrel{f}{\rightarrow} \Z^8\oplus \Z^9\stackrel{g}{\rightarrow} H_0^{\mathfrak F}(\underline{E}SA,R_{\mathbb{C}})\rightarrow 0
\end{equation}

\noindent
where the kernel and the cokernel of $f$ are the desired Bredon homology groups of $G$.

We are led to study the homomorphisms induced in $H^{\mathfrak{F}}_0$ by the inclusions $\mathbf{p2}<\mathbf{p4}$ and $\mathbf{p2}<\mathbf{p6}$. 
As mentioned earlier, we will rely on the results of Section 3.10 in \cite{Flores}. Consider the $\mathbf{p4}$-CW-structure of $\mathbb{R}^2$ that turns the plane into a model for $\underline{E}\mathbf{p4}$. 
As $\mathbf{p2}<\mathbf{p4}$, this structure gives in particular a model for $\underline{E}\mathbf{p2}$, with the induced action from $\mathbf{p2}$. 
With this $\mathbf{p2}$-CW-structure, there are exactly four representatives $v^0_0$, $v^1_0$, $v^2_0$ and $v^3_0$ of equivariant 0-cells, corresponding to the vertices $O$, $P$, $Q$ and $R$. As these vertices are 2-rotation centers, their stabilizers for the $\mathbf{p2}$-action are all isomorphic to $C_2$. In turn, the stabilizers of $e_0^0$, $e_0^1$ and $e_0^2$ are, respectively, isomorphic to $C_4$, $C_2$ and $C_4$. See Figure 1.

Observe that the homomorphism $C_0^{\mathfrak{F}}(\underline{E}\mathbf{p2},R_{\mathbb{C}})\rightarrow C_0^{\mathfrak{F}}(\underline{E}\mathbf{p4},R_{\mathbb{C}})$ induced by the inclusion $\mathbf{p2}<\mathbf{p4}$  is already described as a homomorphism 
$$
 \bigoplus_{i=0}^3 R_{\mathbb{C}}(stab(v_0^i))\rightarrow \bigoplus_{i=0}^2 R_{\mathbb{C}}(stab(e_0^i))
$$ 
defined linearly by induction of representations. 
Also remark that the map $\underline{E}\mathbf{p2}\rightarrow \underline{E}\mathbf{p4}$ induced by the inclusion maps the classes of $0$-equivariant cells $v_0^1$ and $v_0^3$ to the class of $e_0^1$ in $\underline{E}\mathbf{p4}$. 
In turn, the class $v_0^0$ is mapped to the class of $e_0^0$ and the class of $v_0^2$ to the class of $e_0^2$. The inclusion between the stabilizers is trivial in the cases of $v_0^1$ and $v_0^3$, and is given by $C_2<C_4$ in the other two cases.

Now for each $v_0^i$, $0\leq i\leq 3$, we denote by $(\delta_i^1, \delta_i^2)$ the basis for irreducible characters of $R_{\mathbb{C}}(stab(v_0^i))$, with $\delta_i^1$ being the one corresponding to the trivial representation. For the characters $\alpha_j^i$, the identification of the representation rings (appeared in Table 2 in \cite{Flores}) of the previous homomorphism permits to express it as 
$$
\bigoplus_{i=0}^3 (\Z\delta_i^1\oplus\Z\delta_i^2)\rightarrow  \bigoplus_{i=1}^4\Z\alpha_0^i\bigoplus_{i=1}^2\Z\alpha_1^i\bigoplus_{i=1}^4\Z\alpha_2^i.
$$
Observe that explicit bases of $H_0^{\mathfrak{F}}(\underline{E}\mathbf{p2},R_{\mathbb{C}})$ and $H_0^{\mathfrak{F}}(\underline{E}\mathbf{p4},R_{\mathbb{C}})$ are given in Table \ref{Bredon homology}, just replacing $\alpha$ by $\delta$ in the case of $\mathbf{p2}$.

Now using the previous description of the cell structures of $\underline{E}\mathbf{p2}$ and $\underline{E}\mathbf{p4}$, the knowledge of the induction homomorphisms from Table 3 in \cite{Flores} and the description of $\Phi_1$ from sections 3.2 and 3.10 in \cite{Flores}, it can be seen that the matrix associated with the homomorphism $H_0^{\mathfrak{F}}(\underline{E}\mathbf{p2},R_{\mathbb{C}})\rightarrow H_0^{\mathfrak{F}}(\underline{E}\mathbf{p4},R_{\mathbb{C}})$ is:

$$
\begin{pmatrix}
1 & 0 & 0 & 0 & 0 \\
1 & 0 & 0 & 0 & 0 \\
0 & 1 & 0 & 0 & 0 \\
0 & 1 & 0 & 0 & 0 \\
0 & 0 & 1 & 0 & 1 \\
0 & 0 & 0 & 0 & 0 \\
0 & 0 & 0 & 1 & 0 \\
0 & 0 & 0 & 1 & 0
\end{pmatrix}
$$

Recall that here we are using the bases described in the second line of Table \ref{Bredon homology} for the 5-dimensional column space and in the third line of the same table for the 8-dimensional row space. Of course, we follow the ordering there in each of the basis. 


A similar argument can be applied in order to describe the homomorphism $H_0^{\mathfrak{F}}(\underline{E}\mathbf{p2},R_{\mathbb{C}})\rightarrow H_0^{\mathfrak{F}}(\underline{E}\mathbf{p6}, R_{\mathbb{C}})$ induced by inclusion. Let us consider a model $X$ of $\underline{E}{\mathbf{p6}}$ as a model for $\underline{E}\mathbf{p2}$; there will be five classes $v_0^0$, $v_0^1$, $v_0^2$, $v_0^3$ and $v_0^4$ of representatives of $\mathbf{p2}$-equivariant 0-cells in $X$.
The cells $v_0^0$, $v_0^2$ and $v_0^3$ correspond to the points $O$, $P$ and $R$ in the description of Section 3.16 in \cite{Flores}. Moreover, $v_0^1$ corresponds to the middle point of the left-side of the (big) polygon in the corresponding picture of Figure 1, and finally $v_0^4$ to the point marked with a small triangle in the left of the picture.  
Now the map $\underline{E}\mathbf{p2}\rightarrow \underline{E}\mathbf{p6}$ given by the identity sends $v_0^0$ to $e_0^0$, $v_0^1$, $v_0^2$ and $v_0^3$ to $e_0^2$ and $v_0^4$ to $e_0^1$ (following the notation of \cite[Section 3.16]{Flores} for $e_j^i$). By \cite[Section 3.16]{Flores}, the stabilizers of the $v_0^i$ are isomorphic to $C_2$ for $i\leq 3$, while the stabilizer of  $v_0^4$ is trivial. In turn, the stabilizers of $e_0^0$, $e_0^1$ and $e_0^2$  can be found in Section 3.16 in \cite{Flores}, and are respectively isomorphic to $C_6$, $C_3$ and $C_2$.

We denote by $\delta_j^i$ the irreducible characters that generate $R_{\mathbb{C}}(stab(v_0^j))$, and maintain the notation of that Section 3.16 for the irreducible characters which generate $R_{\mathbb{C}}(stab(e_0^i))$. In view of that, the chain homomorphism in Bredon 0-homology 
is
$$ 
\bigoplus_{i=1}^2\Z\delta_0^i\bigoplus_{i=1}^2\Z\delta_1^i\bigoplus_{i=1}^2\Z\delta_2^i\bigoplus_{i=1}^2\Z\delta_3^i\bigoplus \delta_4^1\rightarrow  \bigoplus_{i=1}^6\Z\alpha_0^i\bigoplus_{i=1}^3\Z\alpha_1^i\bigoplus_{i=1}^2\Z\alpha_2^i.
$$
The characters with the superscript 1 are associated to the trivial representation, and the induction homomorphisms can be found in the Table 3 in \cite{Flores}. 
In turn, bases for the Bredon 0-homology groups of \textbf{p2} and \textbf{p6} are provided in Table \ref{Bredon homology}  ($\alpha$ should be replaced by $\delta$ in the case of $\mathbf{p2}$). Observe that according to this the image of $\delta_4^1$ is irrelevant in homology. Putting together all the information above, we obtain the matrix of the homomorphism induced in  $H_0^{\mathfrak{F}}$ in these bases 


$$
\begin{pmatrix}

 -1 &  1 & 0  & 0  & 0 \\
 -1 &  1 & 0  & 0  & 0 \\
 -1 &  1 & 0  & 0  & 0 \\
 0 &  0 & 0  & 0  & 0 \\
 0 &  0 & 0  & 0  & 0 \\
 0 &  0 & 0  & 0  & 0 \\
 0 &  0 & 0  & 0  & 0 \\
 1 &  0 & 0  &  0 & 0 \\
 1 & 0  & 1  & 1  & 1
  \end{pmatrix}
$$

Stacking on top of each other the two components of the Mayer-Vietoris homomorphism $f$ in \eqref{0-degree-SA} and respecting the change of sign in the second component, we obtain

$$
\begin{pmatrix}
 1 & 0 & 0 & 0 & 0 \\
1 & 0 & 0 & 0 & 0 \\
0 & 1 & 0 & 0 & 0 \\
0 & 1 & 0 & 0 & 0 \\
0 & 0 & 1 & 0 & 1 \\
0 & 0 & 0 & 0 & 0 \\
0 & 0 & 0 & 1 & 0 \\
0 & 0 & 0 & 1 & 0 \\
 1 &  -1 & 0  & 0  & 0 \\
 1 &  -1 & 0  & 0  & 0 \\
 1 &  -1 & 0  & 0  & 0 \\
 0 &  0 & 0  & 0  & 0 \\
 0 &  0 & 0  & 0  & 0 \\
 0 &  0 & 0  & 0  & 0 \\
 0 &  0 & 0  & 0  & 0 \\
 -1 &  0 & 0  &  0 & 0 \\
 -1 & 0  & -1  & -1  & -1
  \end{pmatrix}
$$

The invariant factors of the Smith normal form of this matrix are $(1,1,1,1)$.
This implies that the kernel of the connecting homomorphism $f$ from the exact sequence (\ref{0-degree-SA})
is a free abelian group of rank 1, and the cokernel is free abelian of rank 13.  Combining the information extracted from analyzing the exact sequences \eqref{2-degree-SA} and \eqref{0-degree-SA} obtained from the Mayer-Vietoris long exact sequence \eqref{MV-SA}, it is obtained

\begin{Prop}\label{BredonforSA}
The only non-trivial Bredon homology groups of $\underline{E} SA$ are 
$$H_0^{\mathfrak F}(\underline{E} SA,R_{\mathbb{C}})=\mathbb{Z}^{13}, \quad 
H_1^{\mathfrak F}(\underline{E} SA, R_{\mathbb{C}})=\mathbb{Z}, \quad \text{and} \quad
H_2^{\mathfrak F}(\underline{E} SA,R_{\mathbb{C}})=\mathbb{Z}.$$
\end{Prop}
An immediate consequence of the previous proposition and Theorem \ref{BredonK} is

\begin{Cor}\label{LHSforSA}
The equivariant $K$-homology groups of $SA=\mathbb Z^2 \rtimes SL_2(\mathbb Z)$ are 
$$
K_0^{SA}(\underline{E} SA)=\mathbb{Z}^{14}, \quad \text{and} \quad
K_1^{SA}(\underline{E} SA)=\mathbb{Z}.
$$
\end{Cor}

Observe that the auxiliary matrices coming from the computation the Smith normal form provide a basis for the kernel and cokernel of the homomorphism $f$: The kernel is generated by the class of $\delta_3^2-\delta_1^2$ in 
$
H_0^{\mathfrak{F}}(\underline{E} \mathbf{p2}\mathcal{R}_{\mathbb{C}})$.
For the cokernel, let $\{e_1,\ldots e_{17}\}$ denote the union of elements in the bases of $H_0^{\mathfrak{F}}(\underline{E}\mathbf{p4}, R_{\mathbb{C}})$ and 
$H_0^{\mathfrak{F}}(\underline{E}\mathbf{p6}, R_{\mathbb{C}})$ from Table \ref{Bredon homology}, respecting the order. A basis for $H_0^{\mathfrak{F}}(\underline{E}G, R_{\mathbb{C}})$ can be obtained by the images under homomorphism $g$ in \eqref{0-degree-SA} of
$$
\{e_2,e_4,e_6,e_8\} \cup \{e_9,e_{10},e_{11},e_{12},e_{13},e_{14},e_{15},e_{16},e_{17}\}
$$

Clearly, the bases described above for $K_1$ and $K_0$ can be identified with those of $H_1^{\mathfrak{F}}(\underline{E} SA,R_{\mathbb{C}})$ and $H_0^{\mathfrak{F}}(\underline{E} SA,R_{\mathbb{C}})$, respectively, via the Mayer-Vietoris sequence. In particular, since we have all elements of the basis associated with $\mathbf{p6}$ in that of $SA$, we deduce that:

\begin{Cor}
\label{injp6}

Consider the inclusion $i:\mathbf{p6}\hookrightarrow SA$. Then the homomorphism induced by $i$ at the level of $H_0^{\mathfrak{F}}(-, R_{\mathbb{C}})$ is injective, and the homomorphism induced at the level of $H_2^{\mathfrak{F}}(-, R_{\mathbb{C}})$ is an isomorphism.
\end{Cor}

\begin{proof}
Up to a change of sign, the homomorphism induced in $H_0$ is the second component of $g$ in \eqref{0-degree-SA}. As $\{e_9,\ldots,e_{17}\}$ is a basis for $H_0^{\mathfrak{F}}(\underline{E}\mathbf{p6}, R_{\mathbb{C}})$, we are done.

On the other hand, the homomorphism induced in $H_2$ is the second component of the last homomomorphism of the sequence \eqref{2-degree-SA}. This map is described just after the displayed sequence, and it is immediate for this description that it takes $(0,1)$ to a generator of the second Bredon homology of $\underline{E}SA$. So we are done.

\end{proof}

We finish this section with a remark concerning rational Bredon homology.

\begin{Rem}
A rational version of our result for Bredon homology (without explicit generators) can also be obtained by means of a description of the centralizers of finite order elements together with Theorem 3.25 of \cite{MV03}.
\end{Rem}

\subsection{The right-hand side: K-theory of \texorpdfstring{$SA=\Z^2 \rtimes SL_2(\Z)$}{SA=Z\textasciicircum x SL\_Z}}
\label{RHS-AS}

In this section we plan to give a direct proof of the following theorem, which was proved by S. Walters (Theorem 1.3 in \cite{Walters}). In this way, we will describe generators of $K_0(C^*_r(SA))$ as explicit as possible.

\begin{Thm}\label{Walt} $K_0(C^*_r(SA))\simeq\Z^{14}$, and $K_1(C^*_r(SA))\simeq\Z$.
\end{Thm}

We will need the following result of Pimsner \cite{Pimsner} computing the K-theory of group $C^*$-algebras of amalgamated products (important special cases were obtained by Natsume \cite{Natsume}).\footnote{Walters \cite{Walters} also uses the Pimsner-Natsume sequence, but moreover appeals to the Connes-Chern character, while we do not.}

\begin{Thm}\label{PiNa} Let $A\ast_C B$ be an amalgamated product. There exists a 6-terms exact sequence:
\[
\xymatrix{
 K_0(C^*_r(C)) \ar[r] &
 K_0(C^*_r(A))\oplus K_0(C^*_r(B)) \ar[r]&
 K_0(C^*_r(A\ast_C B)) \ar[d]\\
 K_1(C^*_r(A\ast_C B)) \ar[u] &
 K_1(C^*_r(A))\oplus K_1(C^*_r(B)) \ar[l]&
 K_1(C^*_r(C)) \ar[l]
}
\]
\end{Thm}

Using the notations previously introduced, we write $\mathbf{p2}=\Z^2\rtimes C_2$ (with the generator of $C_2$ acting by $\varepsilon=-{\bf 1}_2$), and $\mathbf{p4}=\Z^2\rtimes C_4$ (with the generator of $C_4$ acting by $S$), and $\mathbf{p6}=\Z^2\rtimes C_6$ (with the generator of $C_6$ acting by $R$). Then the amalgamated product decomposition $SL_2(\Z)=C_4\ast_{C_2} C_6$ lifts to
$$SA=\mathbf{p4}\ast_{\mathbf{p2}}\mathbf{p6}.$$
The K-theory rank of $K_i(C^*(\mathbf{pn}))$, for $n=2,4,6$, has been computed by Yang \cite{Yang} and their explicit generators were given by Echterhoff-L\"uck-Phillips-Walters (Examples 3.7(a), (c), (d) in \cite{ELPW}). Note that $K_1=0$ in all three cases.

To describe $K_0$, let us recall some notation from \cite{ELPW}. If $B$ is a finite-index subgroup of $A$, there is a restriction homomorphism $Res_A^B: K_0(C^*(A))\rightarrow K_0(C^*(B))$ obtained by viewing a projective finite type module over $C^*(A)$ as a projective finite type module over $C^*(B)$. If $B$ is moreover normal in $A$, the restriction map commutes with the $A$-action: for $g\in A$, let $\varphi_g$ be the automorphism of $C^*(A)$ induced by conjugation by $g$; then:
\begin{equation}\label{Rescommutes}
    Res_A^B\circ(\varphi_g)_\star=(\varphi_g)_\star\circ Res_A^B
\end{equation}

Also, we denote by $pr$ the quotient-map $K_0(C^*(B))\rightarrow K_0(C^*(B))/\Z[1]$, where $\Z[1]$ is the infinite cyclic subgroup generated by the class of 1 in $K_0$. By Theorem 3.5 in \cite{ELPW}, for $n=2,4,6$ there exists $F_n\in K_0(C^*(\mathbf{pn}))$ such that $pr\circ Res_{\mathbf{pn}}^{\Z^2}(F_n)$ is a generator of $K_0(C^*(\Z^2))/\Z[1]\simeq\Z$. By naturality of the restriction map, we may assume that, for $n=4,6$, we have $Res_{\mathbf{pn}}^{p2}(F_n)=F_2$.

We choose to write $\mathbf{pn}$ as a multiplicative group; so we denote by $1=((0,0),{\bf 1})$ its identity element and, denoting by $u=(1,0),v=(0,1)$ the canonical basis of $\Z^2$, the vector $(k,m)\in\Z^2$ will be denoted by $u^kv^m$. Then:
\begin{enumerate}
\item For $\mathbf{p2}$: the elements $\varepsilon, u\varepsilon, v\varepsilon, uv\varepsilon$ are a maximal set of non-conjugate involutions in $\mathbf{p2}$, so we can form the corresponding spectral projections in the complex group ring of $\mathbf{p2}$:
$$
 \frac{1+\varepsilon}{2}, \quad 
 \frac{1+u\varepsilon}{2}, \quad
 \frac{1+v\varepsilon}{2}, \quad
 \frac{1+uv\varepsilon}{2}.
 $$
Then $K_0(C^*(\mathbf{p2}))=\Z^6$, with a basis consisting of
$$
B_2=\Bigl\{[1],\, 
[\frac{1+\varepsilon}{2}],\,
[\frac{1+u\varepsilon}{2}],\,
[\frac{1+v\varepsilon}{2}],\,
[\frac{1+uv\varepsilon}{2}],\,
F_2\Bigr\}.
$$
\item For $\mathbf{p4}$: the elements $S, uS$ are non-conjugate of order 4, and we may form their spectral projections:

\begin{align*}
    p_0&=\frac{1}{4} (1+S+S^2+S^3) &
    q_0&=\frac{1}{4}(1+uS+(uS)^2+(uS)^3) \\
    p_1&=\frac{1}{4}(1+iS-S^2-iS^3) &
    q_1&=\frac{1}{4}(1+iuS-(uS)^2-i(uS)^3) \\
    p_2&=\frac{1}{4}(1-S+S^2-S^3) &
    q_2&=\frac{1}{4}(1-uS+(uS)^2-(uS)^3)&
    \end{align*}

Also $uS^2$ is an involution in $p4$ which is not conjugate either to $S^2$ or to $(uS)^2$. We form its spectral projection
$$r=\frac{1+uS^2}{2}.$$
Then $K_0(C^*(\mathbf{p4}))\simeq\Z^9$, with a basis consisting of
$$
B_4=\Bigl\{[1],\,
[p_0],\,
[p_1],\,
[p_2],\,
[q_0],\,
[q_1],\,
[q_2],\,
[r],\,
F_4\Bigr\}.
$$
\item For $\mathbf{p6}$, set $\zeta=e^{\frac{\pi}{3}}$: the element $R$ has order 6, and we build again the corresponding spectral projections (for $j=0,1,2,3,4$):
$$p^\dagger_j=\frac{1+\zeta^jR+(\zeta^jR)^2+(\zeta^jR)^3+(\zeta^jR)^4+(\zeta^jR)^5}{6}.$$
The element $uR^2$ has order 3 and is not conjugate to $R^2$, and we form its spectral projections
$$q^\dagger_0=\frac{1+uR^2+(uR^2)^2}{3}, \qquad q_1^\dagger=\frac{1+\zeta^2uR^2 + (\zeta^2uR^2)^2}{3}.$$
Also $uR^3$ is an involution in $p6$ which is not conjugate to $R^3$, and again we construct its spectral projection
$$r^\dagger=\frac{1+uR^3}{2}.$$
Then $K_0(C^*(\mathbf{p6}))\simeq\Z^{10}$, with a basis consisting of
$$B_6=\bigl\{[1],\,
[p^\dagger_0],\,
[p^\dagger_1],\,
[p^\dagger_2],\,
[p^\dagger_3],\,
[p^\dagger_4],\,
[q^\dagger_0],\,
[q^\dagger_1],\,
[r^\dagger],\,
F_6\Bigr\}.$$
\end{enumerate}

\medskip
If $B$ is a subgroup of $A$, we denote by $\iota_B^A:C^*_r(B)\rightarrow C^*_r(A)$ the inclusion.

For $SA=\mathbf{p4}\ast_{\mathbf{p2}} \mathbf{p6}$, the 6-terms exact sequence of Theorem \ref{PiNa} unfolds as:
\begin{equation}\label{PiNaunfold}
0\rightarrow K_1(C^*_r(SA))\stackrel{\partial}{\rightarrow}K_0(C^*(\mathbf{p2}))\stackrel{(\iota_{\mathbf{p2}}^{\mathbf{p4}},\iota_{\mathbf{p2}}^{\mathbf{p6}})}{\longrightarrow} K_0(C^*(\mathbf{p4}))\oplus K_0(C^*(\mathbf{p6}))
\stackrel{\iota_{\mathbf{p4}}^{SA} -\iota_{\mathbf{p6}}^{SA}}{\longrightarrow} K_0(C^*_r(SA))\rightarrow 0
\end{equation}
So we need to analyze the maps $\iota_{\mathbf{p2}}^{\mathbf{pn}}, n=4,6.$

\begin{Lem}\label{RestInd} For $n=4,6$, we have:
\begin{enumerate}
\item[a)] $(pr\circ Res_{\mathbf{pn}}^{\Z^2})(b)=0$ for every $b\in B_n,\;b\neq F_n$.
\item[b)] $(pr\circ Res_{\mathbf{pn}}^{\Z^2})(\iota_{p2}^{pn}(F_2)-[pn:p2]\cdot F_n)=0$.
\end{enumerate}
\end{Lem}

{\bf Proof:} \begin{enumerate}
\item[a)] This follows immediately from Theorem 3.5 in \cite{ELPW}.
\item[b)] By a particular case of lemma 3.2 in \cite{ELPW}: let $C\subset B\subset A$ be finite-index subgroups of a group $A$, with $B\triangleleft A$. Then the following equality holds as maps $K_0(C^*_r(B))\rightarrow K_0(C^*_r(C))$:
$$Res_A^C\circ \iota_B^A=\sum_{[g]\in A/B} Res_B^C\circ (\varphi_g)_*.$$

We apply this with $C=\Z^2, B=\mathbf{p2}$ and $A=\mathbf{pn}\;(n=4,6)$, and compose with the quotient-map $pr$:
\begin{equation}\label{Restriction}
pr\circ Res_{\mathbf{pn}}^{\Z^2}\circ\iota_{\mathbf{p2}}^{\mathbf{pn}}=\sum_{[g]\in \mathbf{pn}/\mathbf{p2}\simeq C_{n/2}}pr\circ Res_{\mathbf{p2}}^{\Z^2}\circ(\varphi_g)_*
\end{equation}
Since $\Z^2\triangleleft \mathbf{p2}$, for $g\in \mathbf{p2}$ we have $Res_{\mathbf{p2}}^{\Z^2}\circ(\varphi_g)_*=(\varphi_g)_*\circ Res_{\mathbf{p2}}^{\Z^2}$ by (\ref{Rescommutes}). Moreover, by Proposition 3.2.5 in \cite{Isely}, the map $(\varphi_g)_*$ is the identity on $K_0(C^*(\Z^2))$ for every $g\in SL_2(\Z)$. So (\ref{Restriction}) becomes
$$pr\circ Res_{\mathbf{pn}}^{\Z^2}\circ\iota_{\mathbf{p2}}^{\mathbf{pn}}=[\mathbf{pn}:\mathbf{p2}](pr\circ Res_{\mathbf{p2}}^{\Z^2}).$$
Applying this to $F_2\in K_0(C^*(p2))$ we see that $(pr\circ Res_{\mathbf{pn}}^{\Z^2})(\iota_{\mathbf{p2}}^{\mathbf{pn}}(F_2))$ is $[\mathbf{pn}:\mathbf{p2}]$ times a generator of $K_0(C^*(\Z^2))/\Z[1]\simeq\Z$, so the result follows by definition of $F_n$.
\hfill$\square$
\end{enumerate}

\begin{Lem}\label{Indtop4} In the bases $B_2$ and $B_4$, the matrix of $\iota_{\mathbf{p2}}^{\mathbf{p4}}:\Z^6\rightarrow\Z^9$ is given by the following $9\times 6$ matrix:
$$\left(\begin{array}{cccccc}1 & 0 & 0 & 0 & 0 & * \\0 & 1 & 0 & 0 & 0 & * \\0 & 0 & 0 & 0 & 0 & * \\0 & 1 & 0 & 0 & 0 & * \\0 & 0 & 0 & 0 & 1 & * \\0 & 0 & 0 & 0 & 0 & * \\0 & 0 & 0 & 0 & 1 & * \\0 & 0 & 1 & 1 & 0 & * \\0 & 0 & 0 & 0 & 0 & 2\end{array}\right)$$
with kernel the subgroup generated by $(0,0,1,-1,0,0)\in\Z^6$.
\end{Lem}

{\bf Proof:} The kernel is easy to compute. To get the matrix, we tackle columns one by one, using $S^2=\varepsilon$:
\begin{itemize}
\item Clear
\item $\iota_{\mathbf{p2}}^{\mathbf{p4}}(\frac{1+\varepsilon}{2})=p_0+p_2$
\item $\iota_{\mathbf{p2}}^{\mathbf{p4}}(\frac{1+u\varepsilon}{2})=r$
\item $\iota_{\mathbf{p2}}^{\mathbf{p4}}(\frac{1+v\varepsilon}{2})=\frac{1+vS^2}{2}$ and $S^{-1}(\frac{1+vS^2}{2})S=r$
\item $\iota_{\mathbf{p2}}^{\mathbf{p4}}(\frac{1+uv\varepsilon}{2})=\frac{1+uvS^2}{2}=\frac{1+(uS)^2}{2}=q_0+q_2$
\item The last column follows from Lemma \ref{RestInd}
\hfill$\square$
\end{itemize}

\begin{Lem}\label{Indtop6} In the bases $B_2$ and $B_6$, the matrix of $\iota_{\mathbf{p2}}^{\mathbf{p6}}:\Z^6\rightarrow\Z^{10}$ is given by the following $10\times 6$ matrix:
$$\left(\begin{array}{cccccc}1 & 0 & 0 & 0 & 0 & * \\0 & 1 & 0 & 0 & 0 & * \\0 & 0 & 0 & 0 & 0 & * \\0 & 1 & 0 & 0 & 0 & * \\0 & 0 & 0 & 0 & 0 & * \\0 & 1 & 0 & 0 & 0 & * \\0 & 0 & 0 & 0 & 0 & * \\0 & 0 & 0 & 0 & 0 & * \\0 & 0 & 1 & 1 & 1 & * \\0 & 0 & 0 & 0 & 0 & 3\end{array}\right)$$
with kernel the subgroup generated by $(0,0,1,-1,0,0)$ and $(0,0,0,1,-1,0)$ in $\Z^6$.
\end{Lem}

{\bf Proof:} Again the kernel is easily obtained. We construct the matrix column by column.
\begin{itemize}
\item Clear.
\item $\iota_{\mathbf{p2}}^{\mathbf{p6}}(\frac{1+\varepsilon}{2})=\frac{1+R^3}{2}=p_0^\dagger+p_2^\dagger+p_4^\dagger$.
\item $\iota_{\mathbf{p2}}^{\mathbf{p6}}(\frac{1+u\varepsilon}{2})=r^\dagger$.
\item $\iota_{\mathbf{p2}}^{\mathbf{p6}}(\frac{1+v\varepsilon}{2})=\frac{1+vR^3}{2}$ and $R^{-1}(\frac{1+vR^3}{2})R=r^\dagger$.
\item $\iota_{\mathbf{p2}}^{\mathbf{p6}}(\frac{1+uv\varepsilon}{2})=\frac{1+uvR^3}{2}$ and $(uR^2)^{-1}(\frac{1+uvR^3}{2})uR^2=r^\dagger$.
\item The last column follows from Lemma \ref{RestInd}.
\hfill$\square$
\end{itemize}

\medskip
{\bf Proof of Theorem \ref{Walt}:} By the exact sequence (\ref{PiNaunfold}), the group $K_1(C^*_r(SA))$ is isomorphic to $\ker(\iota_{\mathbf{p2}}^{\mathbf{p4}},\iota_{\mathbf{p2}}^{\mathbf{p6}})=\ker\iota_{\mathbf{p2}}^{\mathbf{p4}}\cap\ker\iota_{\mathbf{p2}}^{\mathbf{p6}}$. By Lemmas \ref{Indtop4} and \ref{Indtop6}, this is the infinite cyclic subgroup of $\Z^6$ generated by $(0,0,1,-1,0,0)$. We also see that the image $Im(\iota_{\mathbf{p2}}^{\mathbf{p4}},\iota_{\mathbf{p2}}^{\mathbf{p6}})$ is free abelian of rank 5.

On the other hand the map $(\iota_{\mathbf{p2}}^{\mathbf{p4}},\iota_{\mathbf{p2}}^{\mathbf{p6}}):\Z^6\rightarrow\Z^9\oplus\Z^{10}$ is given by the $19\times 6$ matrix obtained by stacking the matrices of $\iota_{\mathbf{p2}}^{\mathbf{p4}},\iota_{\mathbf{p2}}^{\mathbf{p6}}$ on top of each other. Each column in this matrix is a primitive vector in $\Z^{19}$ (for the 6th column, use the fact that 2 and 3 are coprime). So $K_0(C^*_r(SA))\simeq \Z^{19}/Im(\iota_{\mathbf{p2}}^{\mathbf{p4}},\iota_{\mathbf{p2}}^{\mathbf{p6}})$ is free abelian of rank $19-5=14$.
\hfill$\square$

\medskip
From the proof of Theorem \ref{Walt}, we obtain that

\begin{Cor}\label{RHSforSA}  
Let $SA=\Z^2 \rtimes SL_2(\Z)$.
 A basis for $K_0(C^*_r(SA))\simeq\Z^{14}$ is given by
$$\iota_{\mathbf{p6}}^{SA}(B_6)\cup\iota_{\mathbf{p4}}^{SA}(\bigl\{[p_0],[p_1],[q_0],[q_1]\bigr\}).$$
Further, the assembly map $\mu_0^{SA}\colon K_0^{SA}(\underline{E}SA)\rightarrow K_0(C^*_r(SA))$ is an isomorphism.
\end{Cor}

{\bf Proof:} 
    By $Im(\iota_{\mathbf{p2}}^{\mathbf{p4}}\iota_{\mathbf{p2}}^{\mathbf{p6}})=\ker(\iota_{\mathbf{p4}}^{SA}-\iota_{\mathbf{p6}}^{SA})$, 
    we have that the relations defining $(\Z^{9}\oplus\Z^{10})/Im(\iota_{\mathbf{p2}}^{\mathbf{p4}},\iota_{\mathbf{p2}}^{\mathbf{p6}})$ are given by
$$
\iota_{\mathbf{p4}}^{SA}[1]=\iota_{\mathbf{p6}}^{SA}[1]
$$
$$
\iota_{\mathbf{p4}}^{SA}(p_0+p_2)=\iota_{\mathbf{p6}}^{SA}(p_0^\dagger+p_2^\dagger+p_4^\dagger)
$$
$$
\iota_{\mathbf{p4}}^{SA}(r)=\iota_{\mathbf{p6}}^{SA}(r^\dagger)
$$
$$
\iota_{\mathbf{p4}}^{SA}(q_0+q_2)=\iota_{\mathbf{p6}}^{SA}(r^\dagger)
$$
$$
\iota_{\mathbf{p4}}^{SA}(c)=\iota_{\mathbf{p6}}^{SA}(c^\dagger),
$$
where $c$ (resp. $c^\dagger$) denotes the last column of $\iota_{\mathbf{p2}}^{\mathbf{p4}}$ (resp. $\iota_{\mathbf{p2}}^{\mathbf{p6}}$). The conclusion about the basis follows from examination of those relations. To prove the second statement, note that
by the first part, the map $\iota_{\mathbf{p6}}^{SA}: K_0(C^*(\mathbf{p6}))\rightarrow K_0(C^*_r(SA))$ is injective. The same holds on the left-hand side of the BC, by Corollary \ref{injp6}, that can be rephrased by saying that the map $K_0^{\mathbf{p6}}(\underline{E}\mathbf{p6})\rightarrow K_0^{SA}(\underline{E}SA)$ is injective. Since $\mu_{\mathbf{p6}}$ is an isomorphism, by naturality of the assembly map with respect to the inclusion $\mathbf{p6}\hookrightarrow SA$, we have that $\mu_{SA}$ is an isomorphism from $Im(K_0^{\mathbf{p6}}(\underline{E}\mathbf{p6})\rightarrow K_0^{SA}(\underline{E}SA))$ onto $\text{Im}(\iota_{\mathbf{p6}}^{SA})=\Z^{10}$. Moreover, due to the fact that the four remaining generators of $K_0^G(\underline{E}{SA})$ come from $H_0^{\mathfrak F}({SA},R_{\mathbb{C}})$,  and hence are given by characters $\chi$ of suitable finite abelian subgroups $A$ of ${SA}$, we may use the fact that $\mu_{SA}$ maps $\chi$ to the associated projection in the group ring $\C A$ (see Example 2.11 in Part 2 of \cite{MV03}). This completes the proof.
\hfill$\square$


\section{K-homology and K-theory of \texorpdfstring{$GA=\Z^2\rtimes GL_2(\Z)$}{GA=Z\textasciicircum x GL\_2(Z)}}

\subsection{The left hand side: equivariant K-homology of \texorpdfstring{$GA$}{GA}}

\label{BredonGA}

\subsubsection{Preliminaries}

In this section we compute the left-hand side of the Baum-Connes assembly map for the general affine group $GA=\Z^2 \rtimes GL_2(\Z)$. As in the case of $SA$, we will compute Bredon homology with coefficients in the complex representation ring functor $R_{\mathbb{C}}$, and then we deduce its equivariant $K$-homology from the Atiyah-Hirzebruch spectral sequence. By Corollary \ref{3Dmodel} there is  a 3-dimensional model for $\underline{E} GA$, so by Theorem \ref{Mislin} the spectral sequence reduces to an exact sequence. An important observation is that $GA$ decomposes as a push-out of wallpaper groups

$$GA = \mathbf{p4m}\ast_{\mathbf{cmm}} \mathbf{p6m}.$$

Let us elaborate on this decomposition. 
As we recalled in Section \ref{GL2amalgam}, the group $GL_2(\Z)$ can be viewed as an amalgam of the dihedral groups $D_6$ and $D_4$ through a diagonal copy of $D_2$, this decomposition being realized by considering the following explicit subgroups of $GL_2(\mathbb{Z})$
\begin{equation}\label{D6}
D_6=\Big\langle 
\begin{pmatrix}
 0 & 1 \\-1 & 1   
\end{pmatrix}, 
\begin{pmatrix}
 0 & 1 \\ 1 & 0   
\end{pmatrix}
\Big\rangle 
\end{equation}
\begin{equation}\label{D4}
    D_4= \Big\langle 
\begin{pmatrix}
 0 & 1 \\-1 & 0   
\end{pmatrix}, 
\begin{pmatrix}
 0 & 1 \\ 1 & 0   
\end{pmatrix}
\Big\rangle 
\end{equation}

\begin{equation}\label{D2}
    D_2=\Big\langle 
\begin{pmatrix}
 -1 & 0 \\0& -1 
\end{pmatrix}, 
\begin{pmatrix}
 0 & 1 \\ 1 & 0   
\end{pmatrix}
\Big\rangle 
\end{equation}

Moreover, the pullback of the diagram $ D_2 \hookrightarrow GL_2(\Z) \leftarrow  GA$ is isomorphic to $\mathbf{cmm}$, the pullback of $ D_4 \hookrightarrow GL_2(\Z) \leftarrow  GA$ is isomorphic to $\mathbf{p4m}$, and the pullback of $ D_6 \hookrightarrow GL_2(\Z) \leftarrow  GA$ is isomorphic to $\mathbf{p6m}$. Now, the inclusions $\mathbf{p4m}\hookrightarrow GA$ and $\mathbf{p6m}\hookrightarrow GA$ of the pullbacks induce a homomorphism $\mathbf{p4m}\ast \mathbf{p6m}\rightarrow GA$ which is surjective and factors through another homomorphism $\mathbf{p4m}\ast_{\mathbf{cmm}} \mathbf{p6m}\rightarrow GA$. It is not hard to see that the latter homomorphism is an isomorphism.




Our first aim is to compute the Bredon homology groups $H^{\mathfrak F}_*(GA,R_{\mathbb{C}})$. As in the previous case, we will obtain the Bredon groups of $GA$ from that of \textbf{p4m}, \textbf{p6m} and \textbf{cmm} via a Mayer-Vietoris argument.

\subsubsection{Bredon homology and equivariant K-homology of \texorpdfstring{$GA$}{GA}}

As with the case of $SA$, we will intensively use the results and notation in \cite{Flores}, and in particular that of sections 3.9, 3.11 and 3.17, which concern the wallpaper groups of interest here. Remark that all the necessary notation is described at the beginning of its Section 3. We recall in Table \ref{Bredon homology 2} the relevant Bredon homology groups. As we did with the wallpaper groups involved in the structure of $SA$, we briefly explain the notation of the table.

\begin{itemize}

\item For $\mathbf{cmm}$, there are three classes $e_i$, $0\leq i\leq 2$ of 0-equivariant cells whose stabilizers are isomorphic to $D_2$, $D_2$ and $C_2$, respectively. Then $\alpha_i^j$ denotes a generator of the copy of $R_{\C}(D_2)=\Z^4$ corresponding to $e_i$, $i=0,1$, and $\alpha_2^j$ denotes a generator of the copy of $R_{\C}(C_2)$ corresponding to $e_2$.

\item For $\mathbf{p4m}$, there are three classes $e_i$, $0\leq i\leq 2$ of 0-equivariant cells whose stabilizers are isomorphic to $D_4$, $D_4$ and $D_2$, respectively. Then $\alpha_i^j$ denotes a generator of the copy of $R_{\C}(D_4)=\Z^5$ corresponding to $e_i$, $i=0,1$, and $\alpha_2^j$ denotes a generator of the copy of $R_{\C}(D_2)$ corresponding to $e_2$.

\item For $\mathbf{p6}$, there are three classes $e_i$, $0\leq i\leq 2$ of 0-equivariant cells whose stabilizers are isomorphic to $D_6$, $D_3$ and $D_2$, respectively. Then $\alpha_0^j$ denotes a generator of the copy of $R_{\C}(D_6)=\Z^6$ corresponding to $e_0$, $\alpha_1^j$ denotes a generator of the copy of $R_{\C}(D_3)=\Z^3$ corresponding to $e_1$, and $\alpha_2^j$ denotes a generator of the copy of $R_{\C}(D_2)$ corresponding to $e_2$. 

\end{itemize}

In all of the characters the superscript $0$ stands for the trivial representation.

\begin{table}[t]
\begin{center}
\begin{tabular}{| c | c | c | c | c |}
\hline
Group & $H_2^{\mathfrak{F}}$ & $H_1^{\mathfrak{F}}$ & $H_0^{\mathfrak{F}}$ & Basis $H_0^{\mathfrak{F}}$ \\ \hline

$\mathbf{cmm} $  & $ 0 $ & $ 0 $ & $ \Z^6 $ & $ ([\alpha_0^1+\alpha_0^2], [\alpha_0^3], [\alpha_0^4], [\alpha_1^1], [\alpha_1^3], [\alpha_2^2]) $\\ \hline
$\mathbf{p4m} $  & $ 0 $ & $ 0 $ & $ \Z^9 $ & $ ([\alpha_0^4],[\alpha_0^5],[\alpha_1^3],[\alpha_1^4],[\alpha_1^5],[\alpha_2^1],[\alpha_2^2],[\alpha_2^3],[\alpha_2^4]) $\\ \hline
$\mathbf{p6m} $ & $ 0 $ & $ 0 $  & $ \Z^8 $ & $ ([\alpha_0^4],[\alpha_0^5],[\alpha_0^6],[\alpha_1^1],[\alpha_1^3],[\alpha_2^1],[\alpha_2^3],[\alpha_2^4])  $\\ \hline

\end{tabular}
\caption{Bredon homology}
\label{Bredon homology 2}
\end{center}
\end{table}

Due to the decomposition of $GA$ as an amalgam, we may apply the Mayer-Vietoris exact sequence to compute the Bredon homology groups. Thanks to the fact that Bredon homology of wallpaper groups vanishes above dimension 2 and that there is a $3$-dimensional model for $\underline{E}GA$ (Corollary \ref{3Dmodel}), this sequence reduces to 
$$ 
\footnotesize{
\begin{tikzcd}
  &  
  \hskip 13em 0
  \rar 
  &  
  H_3^{\mathfrak F}(\underline{E}GA,R_{\mathbb{C}}) 
   \ar[draw=none]{d}[name=X, anchor=center]{}
 \ar[rounded corners,
            to path={ -- ([xshift=2ex]\tikztostart.east)
                      |- (X.center) \tikztonodes
                      -| ([xshift=-2ex]\tikztotarget.west)
                      -- (\tikztotarget)}]{dll}[at end,above]{} 
                      \\    
  H_2^{\mathfrak F}(\underline{E}\mathbf{cmm},R_{\mathbb{C}})
  \rar
  & 
  H_2^{\mathfrak F}(\underline{E}\mathbf{p4m},R_{\mathbb{C}})\oplus H_2^{\mathfrak F}(\underline{E}\mathbf{p6m},R_{\mathbb{C}})
  \rar
  &
  H_2^{\mathfrak F}(\underline{E}GA,R_{\mathbb{C}})
             \ar[draw=none]{d}[name=X, anchor=center]{}
 \ar[rounded corners,
            to path={ -- ([xshift=2ex]\tikztostart.east)
                      |- (X.center) \tikztonodes
                      -| ([xshift=-2ex]\tikztotarget.west)
                      -- (\tikztotarget)}]{dll}[at end,above]{} 
\\                 
  H_1^{\mathfrak F}(\underline{E}\mathbf{cmm}, R_{\mathbb{C}})
  \rar
  &
  H_1^{\mathfrak F}(\underline{E}\mathbf{p4m}, R_{\mathbb{C}})
  \oplus 
  H_1^{\mathfrak F}(\underline{E}\mathbf{p6m},R_{\mathbb{C}})
  \rar
  &
  H_1^{\mathfrak F}(\underline{E}GA, R_{\mathbb{C}})
  			 \ar[draw=none]{d}[name=X, anchor=center]{}
  \ar[rounded corners,
            to path={ -- ([xshift=2ex]\tikztostart.east)
                      |- (X.center) \tikztonodes
                      -| ([xshift=-2ex]\tikztotarget.west)
                      -- (\tikztotarget)}]{dll}[at end,above]{}
 \\
 H_0^{\mathfrak F}(\underline{E}\mathbf{cmm}, R_{\mathbb{C}})
 \rar
 &
 H_0^{\mathfrak F}(\underline{E}\mathbf{p4m}, R_{\mathbb{C}})
 \oplus 
 H_0^{\mathfrak F}(\underline{E}\mathbf{p6m}, R_{\mathbb{C}})
 \rar
 &
 H_0^{\mathfrak F}(\underline{E} GA, R_{\mathbb{C}})
\rar
&
0
\end{tikzcd}
}
$$
Table \ref{Bredon homology 2} reveals that the Bredon homology groups of these wallpaper groups vanish above degree $0$, therefore 
\begin{equation} \label{LHS}
\footnotesize{
\begin{tikzcd}
&  
  \hskip 13em 0
  \rar 
  &  
H_1^{\mathfrak F}(\underline{E}GA, R_{\mathbb{C}})
  			 \ar[draw=none]{d}[name=X, anchor=center]{}
  \ar[rounded corners,
            to path={ -- ([xshift=2ex]\tikztostart.east)
                      |- (X.center) \tikztonodes
                      -| ([xshift=-2ex]\tikztotarget.west)
                      -- (\tikztotarget)}]{dll}[at end,above]{}
 \\
 H_0^{\mathfrak F}(\underline{E}\mathbf{cmm}, R_{\mathbb{C}})
 \rar
 &
 H_0^{\mathfrak F}(\underline{E}\mathbf{p4m}, R_{\mathbb{C}})
 \oplus 
 H_0^{\mathfrak F}(\underline{E}\mathbf{p6m}, R_{\mathbb{C}})
 \rar
 &
 H_0^{\mathfrak F}(\underline{E} GA, R_{\mathbb{C}})
\rar
&
0
\end{tikzcd}
}
\end{equation}
and in particular we have
\begin{equation} \label{MV}
   0\rightarrow H_1^{\mathfrak F}(\underline{E} GA, R_{\mathbb{C}})\rightarrow \mathbb{Z}^6 \stackrel{f}{\rightarrow} \mathbb{Z}^9 \oplus \mathbb{Z}^8 \stackrel{g}{\rightarrow} 
   H_0^{\mathfrak F}(\underline{E}GA,R_{\mathbb{C}})\rightarrow 0 
\end{equation}

In order to compute the Bredon homology of $GA$, we need to understand the connecting homomorphisms induced in the 0-th group by the inclusions of \textbf{cmm} in \textbf{p4m} and \textbf{p6m}. The arguments will be similar to the ones used in Section \ref{BredonSA}.

We consider the model $X$ of $\underline{E}\mathbf{p4m}$ described in \cite[Section 3.11]{Flores} as a model for $\underline{E}\mathbf{cmm}$ with the induced action (see Figure 1).
As a $\mathbf{cmm}$-CW-complex, there are exactly three representatives $v_0^0$, $v_0^1$ and $v_0^2$ of equivariant 0-cells, corresponding respectively to the vertices $O$, $P$ and $Q$ of $X$. 
The $G$-map $\underline{E}\mathbf{cmm}\rightarrow \underline{E}\mathbf{p4m}$ maps $v_0^i$ to $e_0^i$ for $0\leq i \leq 2$, hence the homomorphism between the 0-chains of \textbf{cmm} and \textbf{p4m} induced by inclusion is given by  
$$
\bigoplus_{i=0}^2 R_{\mathbb{C}}(v_0^i)\rightarrow \bigoplus_{i=0}^2 R_{\mathbb{C}}(e_0^i).
$$ 
Observe that $v_0^0$, $v_0^1$ and $v_0^2$ are identified with $e_0^0$, $e_0^1$ and $e_0^2$ (of Section 3.9 in \cite{Flores}), therefore $stab(v_0^0)\simeq stab(v_0^1)\simeq D_2$, while $stab(v_0^2) \simeq C_2$. In turn, $stab(e_0^0) = stab(e_0^1) \simeq D_4$ in \textbf{p4m}, while $stab(e_0^2)=D_2$. 
Note that the representation ring of the above groups are free abelian (see Table 2 in \cite{Flores}).
Denote by $\delta_j^i$ the irreducible characters generating 
$R_{\mathbb{C}}(stab(v_0^j))$, and by $\alpha_j^i$ that of 
$R_{\mathbb{C}}(stab(e_0^i))$. The above homomorphism between chain complexes can be explicitly described as
$$ 
\bigoplus_{i=1}^4\Z\delta_0^i\bigoplus_{i=1}^4\Z\delta_1^i\bigoplus_{i=1}^2\Z\delta_2^i\rightarrow  \bigoplus_{i=1}^5\Z\alpha_0^i\bigoplus_{i=1}^5\Z\alpha_1^i\bigoplus_{i=1}^4\Z\alpha_2^i.
$$

In general, the characters with superscript 1 denote the ones induced by the trivial representation. In turn, the characters with superscript 5 denote the ones associated with degree 2 representations of $D_4$. 
Observe that for $j=0,1$, the induced character by $\delta_j^i$ via the inclusion $D_2<D_4$ is $\alpha_j^1+\alpha_j^3$ for $i=1$, it is
$\alpha_j^2+\alpha_j^4$ for $i=2$, 
and it is $\alpha_j^5$ for $i=3$ or $i=4$. 
(The induced characters via $C_2<D_2$ are described in Table 3 in \cite{Flores})

Now Table \ref{Bredon homology 2} gives explicit bases of 
$H_0^{\mathfrak{F}}(\underline{E}\mathbf{cmm}, R_{\mathbb{C}})$ and $H_0^{\mathfrak{F}}(\underline{E}\mathbf{p4m},R_{\mathbb{C}})$ (replacing $\alpha$ by $\delta$ in the first case). 
Our considerations about the $\mathbf{cmm}$- and $\mathbf{p4m}$-CW structures of $X$ and the induction homomorphisms, together with the description of $\Phi_1$ from Sections 3.9 and 3.11 of \cite{Flores} imply that the matrix describing  the homomorphism 
$H_0^{\mathfrak{F}}(\underline{E}\mathbf{cmm}, R_{\mathbb{C}})\rightarrow H_0^{\mathfrak{F}}(\underline{E}\mathbf{p4m}, R_{\mathbb{C}})$ is

$$
\begin{pmatrix}
  0 & 0 & 0 & 0 & 0 & 0 \\
  -2 & 1 & 1 & 0 & 0 & 0 \\
  0  & 0 & 0 & 1 & 0 & 0 \\
  0 & 0 & 0 & -1 & 0 & 0 \\
   0 & 0 & 0 & -1 & 1 & 0 \\
   1 & 0 & 0 & 1 & 0 & 0 \\
   1 & 0 & 0 & 1 & 0 & 0 \\
   1 & 0 & 0 & 0 & 0 & 1 \\
   1 & 0 & 0 & 0 & 0 & 1
  \end{pmatrix}
$$

A similar argument applies to the inclusion $\mathbf{cmm}<\mathbf{p6m}$. 
In this case, considering the model $X$ of $\underline{E}{\mathbf{p6m}}$ pictured in Figure 1 as a model for $\underline{E}\mathbf{cmm}$, there are four classes $v_0^0$, $v_0^1$, $v_0^2$ and $v_0^3$ of representatives of $\mathbf{cmm}$-equivariant 0-cells in $X$.
The cells $v_0^0$ and $v_0^2$ correspond to the points $O$ and $R$ in \cite[Section 3.17]{Flores}, while $v_0^3$ corresponds to the center of $X$, and $v_0^1$ to the point $P$ in Section 3.17 of \cite{Flores}. Note that the center of $X$ and $R$ lie in the same orbit under the $\mathbf{p6m}$-action, but not under the $\mathbf{cmm}$-action.

Following the notation of Section 3.17 for the representatives $e_j^i$ of $\mathbf{p6m}$-equivariant 0-cells in $X$, the identity $\underline{E}\mathbf{cmm}\rightarrow \underline{E}\mathbf{p6m}$ maps $v_0^0$ to $e_0^0$, $v_0^1$ to $e_0^1$ and $v_0^2$ and $v_0^3$ to $e_0^2$ (because of the last sentence of the previous paragraph). 
By description given in Section 3 in \cite{Flores}, the stabilizers of $v_0^0$ and $v_0^1$ are isomorphic to $D_2$, the stabilizer of $v_0^2$ is isomorphic to $C_2$ and the stabilizer of $v_0^3$ is trivial.  
In turn, the stabilizers of $e_0^0$, $e_0^1$ and $e_0^2$ are, respectively, isomorphic to $D_6$, $D_3$ and $D_2$.

Denote by $\delta_j^i$ the irreducible characters generating $R_{\mathbb{C}}(stab(v_0^j))$ and by $\alpha_j^i$ the irreducible characters generating $R_{\mathbb{C}}(stab(e_0^i))$, the chain homomorphism at level zero decomposes as

$$ \bigoplus_{i=1}^4\Z\delta_0^i\bigoplus_{i=1}^4\Z\delta_1^i\bigoplus_{i=1}^2\Z\delta_2^i\oplus \Z\delta_3^1\rightarrow  \bigoplus_{i=1}^6\Z\alpha_0^i\bigoplus_{i=1}^3\Z\alpha_1^i\bigoplus_{i=1}^4\Z\alpha_2^i.$$

As before the characters with the superscript 1 are associated to the trivial representation, and $\alpha_5^0$ and $\alpha_6^0$ are used for the 2-dimensional ones.

The only induction that cannot be read off from Table 3 in \cite{Flores} is the one given from $D_2<D_6$. It takes into account the character tables of $D_2$ and $D_6$, and it is defined as follows. 
The induced character by $\delta_0^1$ via $D_2<D_6$ is $\alpha_0^1+\alpha_0^6$, by $\delta_0^2$ is $\alpha_0^2+\alpha_0^5$, by $\delta_0^3$ is $\alpha_0^3+\alpha_0^5$, and by $\delta_0^4$ is $\alpha_0^4+\alpha_0^5$.

Now changing the notation from $\alpha$ to $\delta$ in the case of $\mathbf{cmm}$, Table \ref{Bredon homology 2} provides bases for $H_0^{\mathfrak{F}}(\underline{E}\mathbf{cmm},R_{\mathbb{C}})$ and $H_0^{\mathfrak{F}}(\underline{E}\mathbf{p6m},R_{\mathbb{C}})$ (as in the computation of the homomorphism induced by $\mathbf{p2}<\mathbf{p6}$, the image of $\delta_3^1$ does not contribute). 
Putting together the information above, the matrix associated with the homomorphism and w.r.t. the bases above is

$$
\begin{pmatrix}
  -2 & 1 & 1 & 0 & 0 & 0 \\
  -2 & 1 & 1 & 0 & 0 & 0 \\
  0  & 0 & 0 & 0 & 0 & 0 \\
  0 & 1 & 0 & 0 & 0 & 0 \\
   0 & 1 & 0 & 0 & 0 & 0 \\
   1 & -1 & 0 & 1 & 0 & 0 \\
   2 & -1 & 0 & 0 & 1 & 1 \\
   1 & 0 & 0 & 0 & 0 & 1
  \end{pmatrix}
$$

Stacking on top of each other the two matrices describing the induction in Bredon homology, and respecting the change of sign for the second matrix, we get the matrix associated with the homomorphism $f$ of the Mayer-Vietoris exact sequence \eqref{MV}:
 
$$
\begin{pmatrix}
 0 & 0 & 0 & 0 & 0 & 0 \\
  -2 & 1 & 1 & 0 & 0 & 0 \\
  0  & 0 & 0 & 1 & 0 & 0 \\
  0 & 0 & 0 & -1 & 0 & 0 \\
   0 & 0 & 0 & -1 & 1 & 0 \\
   1 & 0 & 0 & 1 & 0 & 0 \\
   1 & 0 & 0 & 1 & 0 & 0 \\
   1 & 0 & 0 & 0 & 0 & 1 \\
   1 & 0 & 0 & 0 & 0 & 1 \\
  2 & -1 & -1 & 0 & 0 & 0 \\
  2 & -1 & -1 & 0 & 0 & 0 \\
  0  & 0 & 0 & 0 & 0 & 0 \\
  0 & -1 & 0 & 0 & 0 & 0 \\
   0 & -1 & 0 & 0 & 0 & 0 \\
   -1 & 1 & 0 & -1 & 0 & 0 \\
   -2 & 1 & 0 & 0 & -1 & -1 \\
   -1 & 0 & 0 & 0 & 0 & -1

  \end{pmatrix}
$$

The invariant factors of this matrix are $(1,1,1,1,1,1)$, and in particular the Mayer-Vietoris homomorphism is injective. This yields

\begin{Prop}
The only non-trivial Bredon homology group of $GA$ is 
$$H_0^{\mathfrak F}(\underline{E} GA,R_{\mathbb{C}})=\mathbb{Z}^{11}.$$
\end{Prop}
Applying this to Theorem \ref{Mislin} we obtain
\begin{Cor}\label{LHSforGA}
The equivariant $K$-groups of $GA = \mathbb Z^2 \rtimes GL_2(\mathbb Z)$ are $$K_0^{GA}(\underline{E}GA)=\mathbb{Z}^{11} \quad \text{and} \quad K_1^{GA}(\underline{E}GA)=0.
$$
\end{Cor}

A basis for $H_0^{\mathfrak F}(\underline{E} GA,R_{\mathbb{C}})$ can be obtained as the image under the homomorphism $g$ of the sequence \eqref{MV} of a basis of the cokernel of $f$ in the same sequence. Such a concrete basis, obtained as a byproduct of the computation of the Smith normal form, can be given in the following way. Denote by $\{e_1,\ldots e_{17}\}$ the basis of $H_0^{\mathfrak F}(\underline{E}\mathbf{p4m},R_{\mathbb{C}})\oplus H_0^{\mathfrak F}(\underline{E}\mathbf{p6m},R_{\mathbb{C}})$ in the order we have used when computing the previous matrix. Then a basis of 
$H_0^{\mathfrak F}(\underline{E} GA,R_{\mathbb{C}})$ is the image under $f$ of the vectors 
$$\{e_1,e_2-e_4,e_7-e_2,e_9,e_{10},e_{11},e_{12},e_{13},e_{14},e_{16},e_{17}\}.$$

\subsection{The right-hand side: K-theory of \texorpdfstring{$GA=\Z^2\rtimes GL_2(\Z)$}{GA=Z\textasciicircum x GL\_2(Z)}}

In this section we compute K-theory for $C^*_r(GA)$. This will describe the right-hand side of the assembly map for $\mathbb Z^2 \rtimes GL_2(\mathbb Z)$. Combining this with the description we obtained for the left-hand side in Section \ref{BredonGA}, we will be able to reprove that the Baum-Connes assembly map for this group is an isomorphism. 

\subsubsection{Strategy}
	Thanks to the decomposition of $GL_2(\mathbb Z)\simeq D_4\ast_{D_2} D_6$ as an amalgamated free product, its reduced group $C^* $-algebra can be viewed as 
	$$C^*_r(\mathbb Z^2 \rtimes GL_2(\mathbb Z)) \simeq C^*_r(\mathbb Z^2 \rtimes D_4) \ast_{C^*_r(\mathbb Z^2 \rtimes D_2)} C^*_r(\mathbb Z^2 \rtimes D_6).$$ 
    Using the notations for plane wallpaper groups, we will denote the group
	$\mathbb Z^2 \rtimes D_2$ by $\mathbf{cmm}$, the group $\mathbb Z^2 \rtimes D_4$ by $\mathbf{p4m}$ and the group $\mathbb Z^2 \rtimes D_6$ by $\mathbf{p6m}$. 
	Given the above decomposition, the Pimsner 6-term exact sequence (Theorem \ref{PiNa}) will be our main tool. In order to use it, we need to compute $K_i(C^*_r (\mathbf{cmm})), K_i(C^*_r(\mathbf{p4m}))$ and $K_i(C^*_r(\mathbf{p6m}))$. 
	The K-theory rank for these groups has been calculated by Yang \cite{Yang}. In particular, $K_1=0$ for these wallpaper groups.  
    The first goal of this section is to provide a basis for each of these $K_0$-groups.  After that, we plug in the information concerning the bases and analyze the simplified Pimsner
    6-term exact sequence below:
    \begin{equation}\label{simp-pimsner}
    \footnotesize{
	0 \longrightarrow K_1(C^*_r(GA)) \longrightarrow 
	  K_0(C^*_r(\mathbf{cmm})) \stackrel{(\iota_2 ^4, \iota_2^6)}{\longrightarrow} K_0(C^*_r(\mathbf{p4m})) \oplus K_0(C^*_r(\mathbf{p6m}))
	\stackrel{j_4 - j_6}{\longrightarrow} K_0(C^*_r(GA)) \longrightarrow 0,
 }
\end{equation}

    To achieve our first aim, we use that wallpaper groups are amenable and $K_1=0$. Let us first recall a fact that we will need in this section.
    \begin{Rem} \label{remark: minimal proj}
		Let $F$ be a finite group and $\rho \in \hat F$ be an irreducible representation. Associated to $\rho$ and for $\|\xi\|=1$, there is a projection $p_{F, \rho}(f)=\frac{deg(\rho)}{|F|}\langle{\rho(f)\xi, \xi}\rangle \in \mathbb CF$. We will use several times the fact that the class of $p_{F, \rho}$ in $K_0(C^*(F))$ is the image of the class of $\rho$ in $K_0^F(\underline{E}F)$ under the assembly map of $F$: see Example 2.11 in \cite{MV03} for all this. In particular, when $\rho$ runs along $\hat{F}$, the classes of the $p_{F, \rho}$'s form a basis of $K_0(C^*(F))$.
	\end{Rem}

The following lemma is perfectly general:

 \begin{Lem}\label{splitexactsequence} Let $N,H$ be arbitrary discrete groups, with $H$ acting on $N$ by automorphisms. Then, for $i=0,1$, the group $K_i(C^*(H))$ is a direct summand in $K_i(C^*(N\rtimes H))$ (Note that we work here with the full group $C^*$-algebras).
 \end{Lem}

\begin{proof} It is a standard fact (see e.g. Corollary 8.2.2 in \cite{WO}) that, for a split short exact sequence of $C^*$-algebras:
$$0\ra I\ra A \ra A/I \ra 0,$$
the group $K_i(A/I)$ is a direct summand in $K_i(A)$.
We apply this to the group $C^*$-algebra of a semi-direct product: because we work with the full $C^*$-algebras, the split short exact sequence of groups
$$
0\ra N\ra N\rtimes H \stackrel{\curvearrowleft}{\ra} H\ra 0
$$
induces a split short exact sequence of $C^*$-algebras:
$$0\ra I\ra C^*(N\rtimes H)\stackrel{\curvearrowleft}{\ra} C^*(H)\ra 0.$$
This concludes the proof.
	\end{proof}

 Using amenability of $\mathbf{cmm}$, $\mathbf{p4m}$ and $\mathbf{p6m}$ plus the fact that $K_0$ of their $C^*$-algebra is free abelian of finite rank (see Theorem 6.4.3 in \cite{Yang}), we immediately have:
	\begin{Cor}\label{proposition-basis for wallpapergroups}
		Let $G= \mathbb Z^2 \rtimes D_i$ for $i=2,4,6 $ be the wallpaper groups $\mathbf{cmm}$, $\mathbf{p4m}$ and $\mathbf{p6m}$, respectively. Then $K_0(\Cs(D_i))$ appears as a direct summand of $K_0(\Cs_r(G))$. Hence a basis for $K_0(\Cs(D_i))$ can be completed to a basis for the associated wallpaper group.\hfill$\square$
	\end{Cor}
	In the sequel we describe bases for $K_0(C^*_r(\mathbf{cmm})$, $K_0(C^*_r(\mathbf{p4m}))$, and \\ $K_0(C^*_r(\mathbf{p6m}))$ employing the strategy explained in Corollary \ref{proposition-basis for wallpapergroups}. In Section \ref{cmm} (resp. \ref{p4m}, resp. \ref{p6m}) below we will use for the dihedral group $D_n$ (with $n=2,4, 6$) the presentation 
 $$D_n=\langle r,s|r^n=s^2=1, srs=r^{-1}\rangle$$
 with $ s = \begin{pmatrix} 
	0 & 1 \\ 1&  0
	\end{pmatrix}$
 and $r$ given as in Equation \ref{D2} (resp. \ref{D4}, resp. \ref{D6}) above. 

	\subsubsection{Explicit K-theory of \texorpdfstring{$C^*_r(\mathbf{cmm})$}{C*\_r(cmm)}} \label{cmm}

	The dihedral group $D_2$ of order 4 can be presented as
	$$D_2 = \langle r, s \mid r^2=s^2= 1, \,sr s=r\rangle.$$
	As in the previous section, we work with the concrete matrices below as generators of $D_2$
	$$r = \begin{pmatrix} 
	-1 & 0 \\ 0 & -1 
	\end{pmatrix}\quad  
	s = \begin{pmatrix} 
	0 & 1 \\ 1&  0
	\end{pmatrix}$$ 
	Let $\mathbb Z^2 = \langle{u, v}\rangle$, where $u=(1, 0)$ and $v=(0,1)$.
	The wallpaper group $\mathbf{cmm}$ has the following presentation:
	\[
	 \mathbf{cmm}=\langle r, s, u, v \mid r^2=s^2= 1, sr s=r, uv=vu, ru=u^{-1}r, rv=v^{-1}r, su=vs,sv=us\rangle
	\]
	It is known (see Theorem 6.4.3. in \cite{Yang}) that $K_0(C^*_r(\mathbf{cmm}))= \mathbb Z^6$ and $K_1=0$. The goal of this section is to make the K-theory group explicit by providing a basis. Due to Corollary \ref{proposition-basis for wallpapergroups} we need to find a basis for $K_0(C^*(D_2))$ and then complete it to a basis for the bigger group. Understanding the maximal finite subgroups of the wallpaper group will be a main ingredient.
	\begin{Rem} \label{cmm-max.fin}
	The wallpaper group $\mathbf{cmm}$ has three maximal finite subgroups up to conjugacy.	
    Each of these correspond to the stabilizers of the 0-cells (vertices) of the fundamental domain of the wallpaper group described in Section 3 in \cite{Flores}. They consist of two non-conjugate copies of $D_2$ and a copy of $C_2$, which can be concretely identified with
		\begin{itemize}
			\item $\{e, \,r, \,s,\, r s\} \simeq D_2$
			\item $\{e,\, uv^{-1}s,\, rs,\, uv^{-1}r\}\simeq D_2$
			\item $\{e, \,v^{-1}r\} \simeq C_2$
		\end{itemize}
	\end{Rem}
	Given the above description, the concrete representatives for the conjugacy classes of finite order elements in $\mathbf{cmm}$ are
	\begin{gather*} \label{cmm-conj.class}
		e, \,r,\, s,\, rs,\, v^{-1}r,\, uv^{-1}r.
	\end{gather*}
	\begin{Rem} \label{remark: choice for basis}
		Let us explain how we choose candidates for bases for the K-theory of the wallpaper groups. Recall that the K-theory rank is known for these groups. Following Corollary \ref{proposition-basis for wallpapergroups}, we start with a basis for the K-theory of the associated dihedral group. Since $C^*_r(D_i) = \mathbb C D_i$, a basis for $K_0(\mathbb C D_i)$ can be given by the minimal projections associated with the irreducible representations of $D_i$ via the formula given in Remark \ref{remark: minimal proj}. Next we need to complete this basis to a basis of its wallpaper group. Here we must make some choices. Observe that in all these cases the number of conjugacy classes of finite order elements is in accordance with the K-theory rank and in particular with the number of projections we need to choose from each maximal finite subgroup (up to conjugation) of the wallpaper group. For simplicity, if we need one projection, we take the one associated with the trivial representation, and if we need more than one, we choose from the ones associated with the non-trivial representations. The task is then to prove that such sets of projections indeed yield bases for the K-theory of the wallpaper groups.
	\end{Rem}
 \begin{Rem}\label{Green} In the next proof and other subsequent ones, we will appeal to the following fact: if $H$ is a subgroup of a discrete group $G$, then the crossed product $C_0(G/H)\rtimes G$ is isomorphic to $C^*(H)\otimes\mathcal{K}(\ell^2(G/H))$, where $\mathcal{K}(\ell^2(G/H))$ denotes the algebra of compact operators on the Hilbert space $\ell^2(G/H)$: see Theorem 4.30 in \cite{DPW}.
 \end{Rem}

\begin{Thm}\label{cmm-proj}
	Consider the six projections below chosen following Remark \ref{remark: choice for basis}. The set $B_{\mathbf{cmm}}=\{[p_1], \cdots,[p_6]\}$ is a basis for $K_0(C^*_r(\mathbf{cmm}))$.
	\begin{align*}
	p_1 &=\frac{1}{4}(1 + r + s + r s)\\
	p_2 &= \frac{1}{4}(1 + r - s - r s) \\
	p_3 &=\frac{1}{4}(1 - r + s - r s)\\
	p_4 & =\frac{1}{4}(1 - r - s + r s) \\
	p_5 &=\frac{1}{4}(1+uv^{-1}r+uv^{-1}s+ rs)\\
	p_6 &=\frac{1}{2}{(1+v^{-1}r)}
	\end{align*}
	
\end{Thm}

\begin{proof}
	 We need to show that $p_1, \cdots,p_6$ are linearly independent and that together  they generate the $K_0$-group. The idea is to make particular use of those points in the dual group, the torus $\mathbb T^2=\mathbb R^2/\mathbb Z^2$,  whose orbit is not of maximal length so that we can extract information through their stabilizers. By forming the corresponding crossed products we obtain the following homomorphisms. Here we identify $\mathbb T^2$ with $[0,1[^2$ as usual.
	\begin{align*}
	\pi_{\tiny{\begin{pmatrix}
		0\\0
		\end{pmatrix}}} &\colon C^*_r (\mathbf{cmm}) \ra \mathbb C D_2\quad u, v\mapsto 1\\
	\pi_{\tiny{\begin{pmatrix}
		\frac{1}{2}\\\frac{1}{2}
		\end{pmatrix}}} &\colon C^*_r (\mathbf{cmm}) \ra \mathbb C D_2 \quad u,v \mapsto -1\\
	\end{align*}
	together with a homomorphism corresponding to an orbit of length 2
	\begin{align*}
	\pi_{\tiny{\begin{pmatrix}
			0\\\frac{1}{2}
			\end{pmatrix}}}\colon  C^*_r(&\mathbf{cmm})  \ra C(\{a_x, a_y\}) \rtimes D_2 \\
	& u\mapsto 1_{\{a_y\}} - 1_{\{a_x\}}\\
	& v\mapsto 1_{\{a_x\}} - 1_ {\{a_y\}},
	\end{align*}
	where 
	$a_x=
		\begin{pmatrix}
			\frac{1}{2} \\ 0
		\end{pmatrix}$ 
	and 
	$a_y=
		\begin{pmatrix}
			0\\ \frac{1}{2} 
		\end{pmatrix}$.
    The associated semi-direct product decomposition (actually a direct product decomposition in this case) is
    \[
    	D_2 = \{e,r\}\rtimes \{e, s\}.
    \]
    Note that 
    \begin{equation}\label{Morita}
        C(\{a_x, a_y\}) \rtimes D_2 \cong C(D_2/\langle r\rangle)\rtimes D_2\cong M_2(\mathbb C) \otimes \Cs(\mathbb Z_2)
    \end{equation}
    by Remark \ref{Green}, and in particular $K_0(C(\{a_x, a_y\}) \rtimes D_2)=\mathbb Z^2$.
    Consider the projections
    \[
    q_1 = 1_{\{a_x\}}\, \frac{1}{2}(1+r), \qquad q_2 = 1_{\{a_x\}} \, \frac{1}{2}(1-r),
    \]
    where $1_{a_x}$ is a minimal projection in $M_2(\mathbb C)$ and the other component is a minimal projection in $\mathbb C \mathbb Z_ 2$ constructed from the stabilizer in $D_2$ of the point $a_x$.
    The classes of the two projections $q_1,q_2$ form a basis of $K_0(C(\{a_x, a_y\}) \rtimes D_2)$: indeed the classes of the projections $\frac{1}{2}(1+r), \frac{1}{2}(1-r)$ form a basis of $K_0(\Cs(\Z_2))$ which, pulled back through the isomorphism in (\ref{Morita}), gives us a basis 
on the left hand side, consisting precisely of the classes of the projections $q_1$ and $q_2$.
    
	Consider now the homomorphism given by the sum of these three homomorphisms and the induced map in K-theory:
    \[
    	\pi \colon K_0(\Cs(\mathbf{cmm})) \longrightarrow K_0(\mathbb C D_2) \oplus K_0(\mathbb C D_2) \oplus K_0(\mathbb C C_2)= \mathbb Z^{10}
    \]
    The matrix of coordinates of the six projections $p_1,...,p_6$, written column by column in terms of the bases for the three direct summands of $\mathbb Z^{10}$ is given by
    $$
    \begin{pmatrix}
    	1&0&0&0&1&1\\
    	0&1&0&0&0&1\\
    	0&0&1&0&0&0\\
    	0&0&0&1&0&0\\
    	1&0&0&0&1&0\\
    	0&1&0&0&0&0\\
    	0&0&1&0&0&1\\
    	0&0&0&1&0&1\\
    	1&1&0&0&0&1\\
    	0&0&1&1&1&1
    \end{pmatrix}
   $$ 
   Let us denote 
   V = span$\{[p_1], \cdots, [p_6]\} \subseteq K_0(C^*_r(\mathbf{cmm}))$.
   The invariant factors of the Smith normal form of the above matrix are $(1,1,1,1,1,1)$, hence $\pi|_V$ is injective and $\pi(V)\leq \mathbb Z^{10}$ is a direct summand.
   Note that rank $V$ = rank $K_0(C^*_r(\mathbf{cmm}))=6$. 
   This implies that $\exists n\geq 1$ such that $n \cdot K_0(C^*_r(\mathbf{cmm}))\subseteq V$ and already shows that $\pi$ is injective, because $\pi|_V$ is.  
   We want to show that $n=1$, i.e. $V= K_0(C^*_r(\mathbf{cmm}))$. 
   Let $x\in K_0(C^*_r(\mathbf{cmm}))$, 
   and let $\mathbb Z^{10} = \pi(V) \oplus W$, for some $W$.  
   Write $\pi(x) = (x_v, x_w)$ so that $\pi(nx)=(nx_v,nx_w)$. Since $\pi(nx)$ is in $\pi(V)$ we have $nx_w=0$ i.e. $x_w=0$, and this shows that $\pi(K_0(C^*_r(\mathbf{cmm}))=\pi(V)$. As $\pi$ is injective, we have $V= K_0(C^*_r(\mathbf{cmm}))$.
   This completes the proof.   
\end{proof}	

\subsubsection{Explicit K-theory of \texorpdfstring{$C^*_r(\mathbf{p4m})$}{C*\_(p4m)}}	\label{p4m}

The dihedral group $D_4$ of order 8 can be presented as 
$$D_4 = \langle r, s \mid r^4=s^2= 1, \,sr s= r ^{-1}\rangle.$$ 
As with the previous section we take the matrices below as generators for $D_4$. 
$$
	r= 
	\begin{pmatrix}
		0 & 1\\
		-1 & 0
	\end{pmatrix},\quad 
	s = 
	\begin{pmatrix} 
	0 & 1 \\ 1&  0
	\end{pmatrix}
$$ 
Let $\mathbb Z^2 = \langle{u, v}\rangle$, where $u=(1, 0)$ and $v=(0,1)$. The wallpaper group $\mathbf{p4m}$ has the following presentation
\[
    \mathbf{p4m} = \langle r,s, u,v \mid r^4=s^2= 1, sr s= r ^{-1}, uv=vu, ru=v^{-1}r, rv=ur, su=vs, sv=us\rangle
\]
 
It is known (Theorem 6.4.3 in \cite{Yang}) that $K_0(\Cs(\mathbf{p4m}))=\mathbb Z^9$ and $K_1=0$. The goal of this section is to provide a basis for the $K_0$-group. We follow the same procedure as with $\mathbf{cmm}$. We start by describing the maximal finite subgroups.

\begin{Rem}\label{p4m-max.fin}
	The wallpaper group $\mathbf{p4m}$ has three maximal finite subgroups up to conjugation. These subgroups correspond to the stabilizers of the three 0-cells (vertices) of the fundamental domain of the wallpaper group as described in Section 3 in \cite{Flores}. They consist of two non-conjugate copies of $D_4$ and a single copy of $D_2$ which can be identified with

	\begin{itemize}
		\item $\{e, \,r,\, s, \,rs,\, r^2s,\, r^3s,\, r^2, \,r^3\} \cong D_4$
		\item $\{e, \,ur,\, ur^3s,\, uv^{-1}s, \,v^{-1} rs, \,r^2s, \,v^{-1}r^3,\,uv^{-1}r ^{2}\} \cong D_4$
        \item $\{e,\, vr^2,\, vrs,\, r ^3 s\} \cong D_2$
    \end{itemize}
\end{Rem}	

Given the above description, the concrete representatives for conjugacy classes of finite order elements are
\begin{gather*}\label{D4-conj-class}
	{e, \,\,r,\, s, \,rs,\, r^2, \,ur,\, ur^3s, \,uv^{-1}r^{2},\, vr ^2.} 
\end{gather*}

\begin{Thm}\label{p4m-basis}
	Consider the 9 projections below chosen according to Remarks \ref{remark: choice for basis} and \ref{remark: minimal proj}. The set $B_{\mathbf{p4m}} = \{[p_1], \cdots, [p_9]\}$ is a basis for $K_0(C^*_r(\mathbf{p4m}))$.
	\begin{gather} \label{p4m-proj}
	\begin{align*}
	p_1 & =  \frac{1}{8}(1 + r + s + rs + r ^2 s + r^3 s + r^2 + r^3)\\
	p_2 & = \frac{1}{8}(1 + r - s - rs - r ^2 s - r^3 s + r^2 + r^3)\\
	p_3 & = \frac{1}{8}(1 - r + s - rs + r ^2 s - r^3 s + r^2 - r^3)\\
	p_4 & =  \frac{1}{8}(1 - r - s + rs - r ^2 s + r^3 s + r^2 - r^3)\\
	p_5 & = \frac{1}{4}(1 - r^2 + s - r^2s)\\
	p_6 & = \frac{1}{8}(1 + ur - uv^{-1}s - v^{-1}rs - r ^2 s - ur ^3 s + uv^{-1}r ^2 + v^{-1}r^3)\\
	p_7 & = \frac{1}{8}(1 - ur + uv^{-1}s - v^{-1}rs + r^2 s - ur ^3 s + uv^{-1}r^2 - v^{-1}r^3)\\
	p_8 & = \frac{1}{8}(1 - ur - uv^{-1}s + v^{-1}rs - r ^2 s + ur ^3 s + uv^{-1}r ^2 - v^{-1}r^3)\\
	p_9 & = \frac{1}{4}(1 + vr ^2 + vr s + r ^3 s)
	\end{align*}
	\end{gather}
\end{Thm}

\begin{proof}
	We use the same idea of the proof of Theorem \ref{cmm-proj}, namely finding enough homomorphisms through the dual action of $D_4$ on points in $\mathbb T^2$ with small orbit. In this way we will show linear independence and generation. 
	Consider
	\[
		\pi_{\tiny{\begin{pmatrix}	0\\ 0\end{pmatrix}}} \colon C^*_r(\mathbf{p4m}) \longrightarrow \mathbb C D_4 \qquad u,v\mapsto 1	
	\]
	\[
		\pi_{\tiny{\begin{pmatrix}	\frac{1}{2}\\ \frac{1}{2} \end{pmatrix}}} \colon C^*_r(\mathbf{p4m}) \longrightarrow \mathbb C D_4 \qquad u,v\mapsto -1	
	\]
	together with a homomorphism $\pi_3$ constructed from an orbit of length 2
	\begin{align*}
		\pi_ {\tiny{\begin{pmatrix} 0 \\ \frac{1}{2}\end{pmatrix}}}\colon C^*_r(&\mathbf{p4m}) \longrightarrow C(\{a_x, a_y\})\rtimes D_4\\
		& u \mapsto 1_{\{a_y\}} - 1_{\{a_x\}} \\
		& v\mapsto 1_{\{a_x\}} - 1_{\{a_y\}}
	\end{align*}
	where $a_x={\begin{pmatrix} \frac{1}{2} \\ 0\end{pmatrix}}, a_y={\begin{pmatrix} 0 \\ \frac{1}{2}\end{pmatrix}}$.
	\\
	The associated semi-direct product decomposition is
	\[
	  D_4 = \{e, rs, r^2, r^3 s\} \rtimes \{e,s\}
	\]
	hence $C(\{a_x, a_y\})\rtimes D_4 \cong M_2(\mathbb C)\otimes \mathbb CD_2$ by Remark \ref{Green} and in particular $K_0(C(\{a_x, a_y\})\rtimes D_4)=\mathbb Z^4$.
	\\
	Appealing to Remark \ref{remark: minimal proj}, from this we obtain 4 projections in $C(\{a_x, a_y\})\rtimes D_4$ 
	\begin{align*}
		q_1 &= 1_{\{a_y\}} \, \frac{1}{4}(e+ rs + r^2 + r^3s)\\
		q_2 &= 1_{\{a_y\}} \, \frac{1}{4}(e- rs + r^2 - r^3s)\\
		q_3 &= 1_{\{a_y\}} \, \frac{1}{4}(e+ rs - r^2 - r^3s)\\
		q_4 &= 1_{\{a_y\}} \, \frac{1}{4}(e-rs -r^2 + r^3s)
	\end{align*}
	where $1_{\{a_y\}}$ is a minimal projection in $M_2(\mathbb C)$ and the second component is a minimal projection in $\mathbb C D_2$ generated by elements from $D_4$ fixing $a_y$. Exactly as in the proof of Theorem \ref{cmm-proj}, the classes of the projections $q_1,q_2,q_3,q_4$ form a basis of $K_0(C(\{a_x, a_y\})\rtimes D_4)$.
    
    Consider now the homomorphism given by the sum of the homomorphisms  induced in K-theory by $\pi_{\tiny{\begin{pmatrix}	1\\1 \end{pmatrix}}}$, $\pi_{\tiny{\begin{pmatrix}	\frac{1}{2}\\ \frac{1}{2} \end{pmatrix}}}$ and $\pi_3$:
    \[
    	\pi \colon K_0(C^*_r(\mathbf{p4m})) \longrightarrow K_0(\mathbb C D_4)\oplus K_0(\mathbb C D_4)\oplus K_0(\mathbb C D_2) = \mathbb Z^{14} 
    \] 
	Let $V=: \langle{[p_1], \cdots, [p_9]}\rangle\subset K_0(C^*_r(\mathbf{p4m}))$. We construct column by column the matrix associated with $\pi$ showing the image of $[p_i]$ with $i= 1, \cdots, 9$:  
	\[
      \begin{pmatrix}
      1&0&0&0&0&0&0&0&1\\
      0&1&0&0&0&1&0&0&0\\
      0&0&1&0&0&0&1&0&0\\
      0&0&0&1&0&0&0&1&1\\
      0&0&0&0&1&0&0&0&0\\
      1&0&0&0&0&0&1&0&0\\
      0&1&0&0&0&0&0&1&0\\
      0&0&1&0&0&0&0&0&0\\
      0&0&0&1&0&1&0&0&0\\
      0&0&0&0&1&0&0&0&1\\
      1&0&0&1&0&0&0&0&1\\
      0&1&1&0&0&0&0&0&0\\
      0&0&0&0&1&0&0&1&0\\
      0&0&0&0&1&1&1&0&1
      \end{pmatrix}
    \]
    Invariant factors of this matrix are $(1,1,1,1,1,1,1,1,1)$. This together with the fact that K-theory is torsion  free can be used in an exact same manner as in the proof of Theorem \ref{cmm-proj} to show that the set $B_{\mathbf{p4m}}$ is a basis for $K_0(C^*_r(\mathbf{p4m}))$.
\end{proof}

\subsubsection{Explicit K-theory of \texorpdfstring{$C^*_r(\mathbf{p6m})$}{C*\_(p6m)}}	\label{p6m}

The dihedral group $D_6$ of order 12 can be presented as $$D_6 = \langle r, s \mid r^6=s^2= 1, \,sr s= r ^{-1}\rangle.$$ In the sequel we use the following matrices as generators of $D_6$:
$$
	r= 
	\begin{pmatrix}
		0 & 1\\
		-1 & 1
	\end{pmatrix}, \qquad  
	s = 
	\begin{pmatrix} 
		0 & 1 \\ 1&  0
	\end{pmatrix}
$$ 
Let $\mathbb Z^2 = \langle{u, v}\rangle $, where $u=(1,0)$ and $v=(0,1)$. The wallpaper group $\mathbf{p6m}$ has the following presentation:
\[
    \mathbf{p6m}= \langle r,s,u,v \mid r^6=s^2= 1, sr s= r^{-1}, uv=vu, ru=v^{-1}r, rv=uvr, su=vs,sv=us\rangle
\]
It is known that $K_0(\Cs(\mathbf{p6m}))=\mathbb Z^8$ and $K_1=0$ (see Theorem 6.4.3 in \cite{Yang}). The goal of this section is to provide a basis for the $K_0$-group. We follow the same steps as with $\mathbf{cmm}$ and $\mathbf{p4m}$. 

\begin{Rem}\label{p6m,max.fin.}
	The wallpaper group $\mathbf{p6m}$ has three maximal finite subgroups up to conjugation. These subgroups correspond to the stabilizers of 0-cells of the fundamental domain of the wallpaper group described in Section 3 in \cite{Flores}. They consist of $D_6$, $D_3$ and a copy of $D_2$ which can be respectively viewed as
	\begin{itemize}
		\item $\{e, \,r,\, s,\, r s,\, r ^2s,\, r^3 s,\, r ^4s,\, r ^5s,\, r^2,\, r^3,\, r^4, \,r^5\} \cong D_6$ (stabilizer of $(0,0)$);
		\item $\{e, \,ur ^4,\, r s,\, ur ^5s,\, uvr ^3s,\, uvr ^2\} \cong D_ 3$ (stabilizer of $(\frac{2}{3},\frac{1}{3})$);
		\item $\{e,\, uv^{-1}r ^3,\, uv^{-1}s,\, r ^3s\}\cong D_2$ (stabilizer of $(\frac{1}{2},\frac{-1}{2})$).
	\end{itemize}
\end{Rem}

With this notation, the
representatives for conjugacy classes of finite order elements in $\mathbf{p6m}$ are
\begin{gather*}\label{p6m-conj.class}
	e, \,r, \, s, \,rs, \,r^2,\, r^3,\, ur ^4,\, uv^{-1} r ^{3}.
\end{gather*}

\begin{Thm}\label{p6m,proj}
	Consider the following eight projections obtained according to Remarks \ref{remark: choice for basis} and \ref{remark: minimal proj}. The set $B_{\mathbf{p6m}}=\{[p_1], \cdots, [p_8]\}$ is a basis for $K_0(C^*_r(\mathbf{p6m}))$.
	\begin{align*}
	p_1& = \frac{1}{12}(1 + r + s + rs + r^2 s + r ^3s + r ^4 s + r ^5 s + r ^2 + r ^3 + r ^4 + r^5)\\
	p_2 & = \frac{1}{12}(1 + r - s - rs - r^2 s - r ^3s - r ^4 s - r ^5 s + r ^2 + r ^3 + r ^4 + r^5)\\
	p_3 & = \frac{1}{12}(1 - r + s - rs + r^2 s - r ^3s + r ^4 s - r ^5 s + r ^2 - r ^3 + r ^4 - r^5)\\
	p_4 & = \frac{1}{12}(1 - r - s + rs - r^2 s + r ^3s - r ^4 s + r ^5 s + r ^2 - r^3 + r ^4 - r^5)\\
	p_5 & = \frac{1}{6}(1 + \frac{1}{2}r + s + \frac{1}{2}rs - \frac{1}{2}r ^2s - r^3 s  - \frac{1}{2}r ^4 s + \frac{1}{2}r ^5 s - \frac{1}{2}r ^2 - r ^3 - \frac{1}{2}r ^4 + \frac{1}{2}r ^5) \\
	p_6 & = \frac{1}{6}(1 - \frac{1}{2}r + s - \frac{1}{2}rs - \frac{1}{2}r ^2s + r^3 s  - \frac{1}{2}r ^4 s - \frac{1}{2}r ^5 s - \frac{1}{2}r ^2 + r^3 - \frac{1}{2}r^4 - \frac{1}{2}r^5) \\
	p_7 & = \frac{1}{6}(1 + ur^4 + rs + ur^5s + uvr ^3s + uvr^2)\\
	p_8 & = \frac{1}{4}(1 + uv^{-1}r^3 + uv^{-1}s + r^3 s)
 \end{align*}
 In terms of Remark \ref{remark: minimal proj}, the projections $p_5,p_6$ are obtained from the 2-dimensional representations $\rho_\pm$ of $D_6$ defined by $\rho_\pm(s)=\begin{pmatrix} 
 1 & 0\\
 0 & -1
 \end{pmatrix}$ and $\rho_+(r)=\begin{pmatrix} 
 \frac{1}{2} & -\frac{\sqrt{3}}{2} \\
 \frac{\sqrt{3}}{2} & \frac{1}{2}
 \end{pmatrix}$, $\rho_-(r)=\begin{pmatrix} 
 -\frac{1}{2} & -\frac{\sqrt{3}}{2} \\
 \frac{\sqrt{3}}{2} & -\frac{1}{2}
 \end{pmatrix}$, by taking the coefficient on the first basis vector.
\end{Thm}
\begin{proof}
We shall show that $p_1, \cdots, p_8$ are linearly independent and generate the $K_0$-group.  We follow the same  strategy as in the two previous cases, and in particular we make use of finite quotients.
Consider 
$$
  \pi_1 \colon C^*_r(\mathbf{p6m}) \longrightarrow \mathbb  C D_6 \qquad u, v \mapsto 1
$$ 
The second homomorphism is constructed using a point in the dual torus whose orbit has length 2.
Consider
\begin{align*}
	\pi_2 &\colon C^*_r(\mathbf{p6m}) \longrightarrow C(\{a,b\})\rtimes D_6 \\
	&u,v \mapsto \frac{1}{2}(-1+\sqrt{3}\,i)1_{\{a\}}+ \frac{1}{2}(-1-\sqrt{3}\,i)1_{\{b\}}	
\end{align*}
where 
$
a=\begin{pmatrix}
\frac{2}{3}\\\frac{1}{3}
\end{pmatrix}
$
and
$
b=\begin{pmatrix}
\frac{1}{3}\\\frac{2}{3}
\end{pmatrix}.
$

The associated semi-direct product decomposition is
$$
	D_6 = \{e, r^2, r^4, rs, r^3s, r^5s\}\rtimes \{e,r^3\}\cong D_3\times \mathbb{Z}_2,
$$
so $C(\{a,b\})\rtimes D_6 =C(D_6/D_3)\rtimes D_6\cong M_{2}(\mathbb C) \otimes \Cs(D_3)$ by Remark \ref{Green}, and in particular $K_0(C(\{a,b\})\rtimes D_6)= \mathbb Z^3$. 
From this we obtain the three projections below
\begin{align*}
	q_1 &= 1_{\{a\}} \,\frac{1}{6}(1 + r^2 + r^4 + rs + r^3s + r^5s)\\
	q_2 &= 1_{\{a\}} \, \frac{1}{6}(1 + r^2 + r^4 - rs - r^3s - r^5s )\\
	q_3 &= 1_{\{a\}} \, \frac{1}{3}(1 + rs - \frac{1}{2}r^2 - \frac{1}{2}r^4 - \frac{1}{2}r^3s- \frac{1}{2}r^5s),
\end{align*}
constructed via minimal projections in each component using the stabilizers of the point $a$ in $D_6$. Exactly as in the proof of Theorem \ref{cmm-proj}, the classes of $q_1,q_2,q_3$ form a basis of $K_0(C(\{a,b\})\rtimes D_6)$.

The last homomorphism is constructed using a point in the dual torus whose orbit has length 3. 
\begin{align*}
	\pi_3 \colon C^*_r(\mathbf{p6m})& \longrightarrow C(\{a_x, b_y,b\})\rtimes D_6\\
	                       & u\mapsto 1_{\{ a_y\}} - 1_{\{a_x,b\}} \\
	                       & v\mapsto 1_{\{ a_x\}} - 1_{\{a_y,b\}},
\end{align*}
where $ a_x=\begin{pmatrix} \frac{1}{2} \\ 0 \end{pmatrix}$,
$ a_y=\begin{pmatrix} 0 \\ \frac{1}{2}\end{pmatrix}$, and
$ b=\begin{pmatrix} \frac{1}{2} \\ \frac{1}{2}\end{pmatrix}$.

The stabilizer of $b$ is $\{e,s,r^3,r^3s\}\cong D_2$ so, by Remark \ref{Green}:
\[
	C(\{a_x, a_y, b\})\rtimes D_6 = C(D_6/D_2)\rtimes D_6\cong M_3(\mathbb C)\otimes \Cs(D_2).
\]
and in particular $K_0(C(\{a_x, a_y, b\})\rtimes D_6) = \mathbb Z^4$.
Using Remark \ref{remark: minimal proj}, the above description gives projections
\begin{align*}
	q_1 &= 1_{\{b\}} \, \frac{1}{4}(1+r^3+s+sr^3)\\
	q_2 &= 1_{\{b\}} \, \frac{1}{4}(1+r^3-s-sr^3)\\
	q_3 &= 1_{\{b\}} \, \frac{1}{4}(1-r^3+s-sr^3)\\
	q_4 &= 1_{\{b\}} \, \frac{1}{4}(1-r^3-s+sr^3),
\end{align*}
where $1_{\{b\}}$ is a minimal projection in $M_3(\mathbb C)$ and the other component is a minimal projection in $\mathbb C D_2$ generated by elements in the stabilizer of the point $b$. Once more, exactly as in the proof of Theorem \ref{cmm-proj}, the classes of $q_1,q_2,q_3,q_4$ form a basis of $K_0(C(\{a_x, a_y, b\})\rtimes D_6)$.
\\
Now consider the map 
$$
\pi \colon K_0(C^*_r(\mathbf{p6m})) \longrightarrow K_0(\mathbb C D_6) 
\oplus K_0(\mathbb C D_3)\oplus 
K_0(\mathbb C D_2)=\mathbb Z^{13} 
$$
defined by homomorphisms $\pi_i$ where $i=1,2,3$.
Using the descriptions the image of the eight candidate projections under $\pi$, we obtain the matrix
\[
\begin{pmatrix}
 1& 0 &0& 0& 0& 0& 1 &1\\
 0 &1 &0 &0 &0 &0 &0 &0\\
 0 &0 &1 &0 &0 &0 &0 &0\\
 0&0&0&1&0&0&1&0\\
 0&0&0&0&1&0&0&0\\
 0&0&0&0&0&1&0&1\\
 1&0&0&1&0&0&0&1\\
 0&1&1&0&0&0&0&0\\
 0&0&0&0&1&1&1&1\\
 2&-1&0&0&0&1&1&1\\
 -1&2&0&0&0&1&0&0\\
 0&0&2&-1&1&0&0&1\\
 0&0&-1&2&1&0&1&1	
\end{pmatrix}
\]
The invariant factors of this matrix are $(1,1,1,1,1,1,1,1)$. This fact and the torsion-freeness of K-theory can be used in the same way as in the proof of Theorem \ref{cmm-proj} to show that the set $B_{\mathbf{p6m}}$ is a basis for $K_0(C^*_r(\mathbf{p6m}))$.
\end{proof}

\subsubsection{K-theory of \texorpdfstring{$C_r^*(GA)$}{C*\_r(GA)}}

Plugging in the information we have obtained so far concerning the K-theory and bases for three wallpaper groups into the Pimsner 6-term exact sequence \cite{Pimsner} we obtain the following exact sequence:
\begin{equation}\label{pimsner}
	0 \longrightarrow K_1(C^*_r(GA)) \longrightarrow 
	\mathbb {Z}^6 \stackrel{(\iota_2 ^4, \iota_2^6)}{\longrightarrow} \mathbb{Z}^9 \oplus \mathbb {Z}^8 
	\stackrel{j_4 - j_6}{\longrightarrow} K_0(C^*_r(GA)) \longrightarrow 0,
\end{equation} where the connecting map $(\iota_2 ^4, \iota_2^6)$ is induced from the inclusion of $\mathbf{cmm}$ into $\mathbf{p4m}$ and $\mathbf{p6m}$  and the connecting map $j_4 - j_6$ is induced from the inclusion of $\mathbf{p4m}$ and $\mathbf{p6m}$ into $\mathbb {Z}^2 \rtimes GL_2(\mathbb Z)$. We need to understand the kernel and the cokernel of the middle homomorphism, and this will allow us to find an explicit basis of K-theory of $C_r^*(\mathbb Z^2 \rtimes GL_2(\mathbb Z))$. For later reference let us renumber the projections that belongin to the bases $B_{\mathbf{4pm}}$ and $B_{\mathbf{p6m}}$ by $[p_1], \cdots, [p_{17}]$, where $p_i$ for $i=1, \cdots, 9$ belong to $B_{\mathbf{p4m}}$ and the rest come from $B_{\mathbf{p6m}}$. 
 	
\begin{Thm}\label{RHSforGA}
	The K-theory for the group $GA = \mathbb Z^2 \rtimes GL(2, \mathbb Z))$ is 
 $$K_0 (C_r^*(\mathbb Z^2 \rtimes GL_2(\mathbb Z))) = \mathbb Z^{11} \quad \text{and} \quad K_1(C_r^*(\mathbb Z^2 \rtimes GL_2(\mathbb Z))) = 0.$$
 A basis for $K_0$ is given by $\{[p_1], \cdots, [p_6], [p_8], [p_{9}], [p_{14}], [p_{15}],  [p_{16}]\}$.
\end{Thm}
\begin{proof}
We describe the homomorphisms induced in K-theory by the inclusions $\iota_2 ^4$ of $\mathbf{cmm}$ into $\mathbf{p4m}$ and $\iota_{2}^{4}$ of $\mathbf{cmm}$ into $\mathbf{p6m}$. If we consider the generating set $B_{\mathbf{cmm}}$ from Theorem \ref{cmm-proj} and interpret its elements in terms of elements of $B_{\mathbf{p4m}}$ from Theorem \ref{p4m-basis}, we obtain

\begin{align*}
	\iota_{2}^{4}(p_1)&=\iota_{2}^{4}(\frac{1}{4}(1 + r_2 + s + r_2s))= \frac{1}{4}(1+ r_4^2 + s + r_4 ^2s)\\
	\iota_{2}^{4}(p_2)&= \iota_{2}^{4}(\frac{1}{4}(1 + r_2 - s - r_2s))= \frac{1}{4}(1+ r_4^2 - s - r_4 ^2s)\\
	\iota_{2}^{4}(p_3)&= \iota_{2}^{4}(\frac{1}{4}(1 - r_2 + s - r_2s))= \frac{1}{4}(1 - r_4^2 + s - r_4 ^2s)\\
	\iota_{2}^{4}(p_4)&=\iota_{2}^{4}(\frac{1}{4}(1 - r_2 - s + r_2s))= \frac{1}{4}(1 - r_4^2 - s + r_4 ^2s)\\
	\iota_{2}^{4}(p_5)&=\iota_{2}^{4}(\frac{1}{4}(1+uv^{-1}r_2+uv^{-1}s+ r_2s))= \frac{1}{4}(1+uv^{-1}r_4^2+uv^{-1}s+ r_4^2s)\\
	\iota_{2}^{4}(p_6)&=\iota_{2}^{4}(\frac{1}{2}{1+v^{-1}r_2})=\frac{1}{2}(1+v^{-1}r_4^2).
\end{align*}

Therefore in the bases of $\mathbf{cmm}$ and $\mathbf{p4m}$ the homomorphism $\iota_{2}^{4} \colon \mathbb Z^6 \ra \mathbb Z^9$ is given by 
$$\iota_2^4=
\begin{pmatrix}

1 & 0 & 0 & 0 & 1 & -1\\
0 &1 &0 &0 &-1 &1 \\
1 &0 &0 &0 &1& 0\\
0 &1 &0& 0& -1& 0 \\
0 &0 &1& 1& 0& 0\\
0 &0& 0& 0& 1& 0\\
0 &0 &0 &0 &0& 1 \\
0 &0& 0& 0& 1& -1\\
0& 0& 0& 0& 0& 2
\end{pmatrix}
$$	
Next we describe the inclusion $\iota_2 ^6$ of $\mathbf{cmm}$ into $\mathbf{p6m}$. In view of $B_{\mathbf{cmm}}$ and $B_{\mathbf{\mathbf{p6m}}}$ from \ref{cmm-proj} and \ref{p4m-proj}, respectively, we obtain
\begin{align*}
	\iota_{2}^{6}(p_1)&=\iota_{2}^{6}(\frac{1}{4}(1 + r_2 + s + r_2s))= \frac{1}{4}(1+ r_6^3 + s + r_6 ^3s)\\
	\iota_{2}^{6}(p_2)&= \iota_{2}^{6}(\frac{1}{4}(1 + r_2 - s - r_2s))= \frac{1}{4}(1+ r_6^3 - s - r_6 ^3s)\\
	\iota_{2}^{6}(p_3)&= \iota_{2}^{6}(\frac{1}{4}(1 - r_2 + s - r_2s))= \frac{1}{4}(1 - r_6^3 + s - r_6^3s)\\
	\iota_{2}^{6}(p_4)&=\iota_{2}^{6}(\frac{1}{4}(1 - r_2 - s + r_2s))= \frac{1}{4}(1 - r_6^3 - s + r_6^3s)\\
	\iota_{2}^{6}(p_5)&=\iota_{2}^{6}(\frac{1}{4}(1+uv^{-1}r_2+uv^{-1}s+ r_2s))= \frac{1}{4}(1+uv^{-1}r_6^3+uv^{-1}s+ r_6^3s)\\
	\iota_{2}^{6}(p_6)&=\iota_{2}^{6}(\frac{1}{2}{1+v^{-1}r_2})=\frac{1}{2}(1+v^{-1}r_6^3).
\end{align*}

Hence, in the bases of $\mathbf{\mathbf{cmm}}$ and $\mathbf{\mathbf{p6m}}$ the homomorphism $\iota_2^6 \colon \mathbb Z^6 \ra \mathbb Z^8$ is given by 
$$\iota_2^6=
	\begin{pmatrix}
		    
		 1 & 0 & 0 & 0 & 0 & -1\\
		 0 &1 &0 &0 &0 &1 \\
		 0 &0 &1 &0 &0& 0\\
		 0 &0 &0&1& 0& 0\\
		 0 &0& 1& 1&0& 0\\
		 1 &1 &0 &0 &0& 0 \\
		 0 &0& 0& 0& 0& 0\\
	     0& 0& 0& 0& 1& 2
	\end{pmatrix}
$$	
Stacking these two matrices on top of each other it is obtained the following $17 \times 6$ matrix, that describes the homomorphism $(\iota_2^4, \iota_2^6)$
$$
	\begin{pmatrix}
		 1 & 0 & 0 & 0 & 1 & -1\\
		 0 &1 &0 &0 &-1 &1 \\
		 1 &0 &0 &0 &1& 0\\
		 0 &1 &0& 0& -1& 0 \\
		 0 &0 &1& 1& 0& 0\\
		 0 &0& 0& 0& 1& 0\\
		 0 &0 &0 &0 &0& 1 \\
		 0 &0& 0& 0& 1& -1\\
		 0& 0& 0& 0& 0& 2\\
		 -1 & 0 & 0 & 0 & 0 & 1\\
		 0 &-1 &0 &0 &0 &-1 \\
		 0 &0 &-1 &0 &0& 0\\
		 0 &0 &0& -1& 0& 0\\
		 0 &0& -1& -1&0& 0\\
		 -1 &-1 &0 &0 &0& 0 \\
		 0 &0& 0& 0& 0& 0\\
		 0& 0& 0& 0& -1& -2
	\end{pmatrix}
$$
The invariant factors of this matrix are $(1,1,1,1,1,1)$, so the middle connecting map in the Pimsner exact sequence is injective and in particular $K_1=\Ker (\iota_2^4, \iota_2^6)=0$.
The cokernel of this matrix describes the $K_0$-group. Note that $rank(\iota_2^4)= 5$ (as $\Ker(\iota_2^4)$ is generated by $(0,0,1,-1,0,0)$) and $rank(\iota_2^6)= 5$ (as $\Ker(\iota_2^6)$ is generated by $(1,-1,0,0,-2,1)$). 

So, denoting the elements of the union of bases of $K_0(C^*_r(\mathbf{p4m}))$ and of $K_0(C^*_r(\mathbf{p6m}))$ by $\{[p_1], \cdots, [p_{17}]\}$, to get a basis of the cokernel of $(\iota_2^4, \iota_2^6)$ we may take all elements of the basis of $K_0(C^*_r(\mathbf{p4m}))$ but one, and all the elements of the basis of $K_0(C^*_r(\mathbf{p6m}))$ but five. So a basis of $K_0(\Cs_r(GA))$ is given for example by:
\[
	\bigl\{
	  [p_1], \cdots, [p_6],[p_8],[p_9],[p_{14}],[p_{15}],[p_{16}]
	\bigr\}
\]
This finishes the proof.
\end{proof}	
	
	\begin{Thm} \label{GL,iso}
		The assembly map of $GA = \mathbb Z^2 \rtimes GL_2(\mathbb Z)$ is an isomorphism.
	\end{Thm}
	\begin{proof}
		The two exact sequences obtained in (\ref{simp-pimsner}) and (\ref{LHS}) yield
 
		\begin{equation*}
  \footnotesize{
		\xymatrix{
	H_0^{\mathfrak F}(\underline{E}\mathbf{cmm}, R_{\mathbb{C}}) 
\ar[r]\ar[d]^{\mu_0^{cmm}}&
  H_0^{\mathfrak F}(\underline{E}\mathbf{p4m}, R_{\mathbb{C}})
 \oplus 
 H_0^{\mathfrak F}(\underline{E}\mathbf{p6m}, R_{\mathbb{C}})
\ar[r]\ar[d]^{\mu_0^{p4m}\oplus \mu_0^{p6m}}&
 H_0^{\mathfrak F}(\underline{E} GA, R_{\mathbb{C}})
\ar[r]\ar[d]^{\mu_0^{GA}}&
0
		\\
K_0(C^*_r(\mathbf{cmm})) \ar[r]& 
K_0(C^*_r(\mathbf{p4m})) \oplus K_0(C^*_r(\mathbf{p6m})) \ar[r]&
	K_0(C^*_r(GA)) \ar[r]& 0,
 }}
		\end{equation*}
		which is commutative thanks to functoriality of the assembly maps. Note that the Baum-Connes assembly map for wallpaper groups is an isomorphism. Hence a suitable version of the 5-lemma implies that $\mu_0^{GA}$ is an isomorphism. 
\end{proof}

\section{K-homology and K-theory for \texorpdfstring{$\Z^2\rtimes\Gamma(2)$}{Z\textasciicircum 2 x Gamma(2)}}

It is known that $\Gamma(2)$ is generated by $\varepsilon=\left(\begin{array}{cc}-1 & 0 \\0 & -1\end{array}\right), T^2=\left(\begin{array}{cc}1 & 2 \\0 & 1\end{array}\right), U^2=\left(\begin{array}{cc}1 & 0 \\2 & 1\end{array}\right)$ (see e.g. Theorem 3.1 in \cite{Conrad}). Write $S$ for the {\it Sanov subgroup}, i.e. the subgroup generated by $T^2$ and $U^2$. An easy ping-pong argument (see e.g. \cite{Wiki}) shows that $S$ is free on $T^2$ and $U^2$. Observe that $S$ is a subgroup of index 2 in $\Gamma(2)$; and $\Gamma(2)=S\times\langle\varepsilon\rangle$. 
W denote $G=\Z^2\rtimes\Gamma(2)$, so we have
$$G\simeq (\Z^2\rtimes \langle\varepsilon\rangle)\rtimes S=\mathbf{p2}\rtimes S.$$

\subsection{The left-hand side: equivariant K-homology}

In order to compute equivariant K-homology for this group we will make use of Mart\'{i}nez' spectral sequence introduced in \cite{Martinez}. This is  an analogue in Bredon homology of the Serre spectral sequence.

Consider the extension $0\rightarrow \mathbf{p2}\rightarrow G\rightarrow S \rightarrow 0$ defining $G$ as a semi-direct product. Denote by $\mathfrak{F}_{\mathbf{p2}}$ the family of finite subgroups of $\mathbf{p2}$. Since $S$ is the free group on 2 generators, the family of finite subgroups of $S$ only contains the trivial subgroup. Hence the same argument as in Section 5.2 in \cite{FPV17} applies, implying that the $E^2$-page of the spectral sequence simplifies to 
$$
E^2_{p,q}=H_p(S,H_q^{\mathfrak{F}_{\mathbf{p2}}}(\underline{E}\mathbf{p2},R_{\mathbb{C}}).
$$
Recall that the ordinary homology of $S$ is concentrated in degrees $0$ and $1$, with values $H_0(S,\Z)=\Z$ and $H_1(S,\Z)=\Z^2$. Moreover, appealing to  Table \ref{Bredon homology}, the only nontrivial Bredon homology groups of $\mathbf{p2}$ are $H_0^{\mathfrak{F}_{\mathbf{p2}}}(\underline{E}\mathbf{p2},R_{\mathbb{C}})=\Z^5$ and $H_2^{\mathfrak{F}_{\mathbf{p2}}}(\underline{E}\mathbf{p2},R_{\mathbb{C}})=\Z$. 
This in particular implies that only four abelian groups in the $E^2$-page can be non-trivial. Let us compute them.

\begin{itemize}

\item $E^2_{0,0}=H_0(S,H_0^{\mathfrak{F}_{\mathbf{p2}}}(\underline{E}\mathbf{p2},R_{\mathbb{C}}))=H_0(S,\Z^5)=\Z^5$

\item $E^2_{1,0}=H_1(S,H_0^{\mathfrak{F}_{\mathbf{p2}}}(\underline{E}\mathbf{p2},R_{\mathbb{C}}))=H_1(S,\Z^5)=\Z^2\otimes \Z^5=\Z^{10}$

\item $E^2_{0,2}=H_0(S,H_2^{\mathfrak{F}_{\mathbf{p2}}}(\underline{E}\mathbf{p2},R_{\mathbb{C}}))=H_0(S,\Z)=\Z$

\item $E^2_{1,2}=H_1(S,H_2^{\mathfrak{F}_{\mathbf{p2}}}(\underline{E}\mathbf{p2},R_{\mathbb{C}}))=H_1(S,\Z)=\Z^2$

\end{itemize}

Note that to describe $E^2_{1,0}$, we appeal to the Universal Coefficient Theorem (see Theorem 3A.3 in \cite{Hatcher}), together with the fact that the groups involved are torsion-free.

Because of the existence of a 3-dimensional model for $\underline{E}G$, the differential $d_2$ in this spectral sequence vanishes, so the sequence collapses at the second page $E^2$, and $E^2=E^{\infty}$. 

Now we have all the ingredients to compute the equivariant $K$-homology of $\underline{E}G$; as in the two previous cases, we appeal to the reduced version of Atiyah-Hirzebruch spectral sequence of Theorem \ref{Mislin}.

\begin{equation} \label{Gamma(2)}
\xymatrix@1{ 
0 \ar[r] & \Z^{10} \ar[r] & K_1^G(\underline{E}G) \ar[r] & \Z^2 \ar[d] 
\\
0  & \Z \ar[l] & K_0^G(\underline{E}G) \ar[l] &
\Z^5 \ar[l]_ -{\Lambda_0(G)}
}
\end{equation}

It is established in \cite[page 52]{MV03} that for any group $\Gamma$ the kernel of the homomorphism $\Lambda_0(\Gamma):H_0^{\mathfrak F}(\Gamma,R_{\mathbb{C}})\rightarrow K_0^{\Gamma}(\underline{E}\Gamma)$ is a torsion group. In our case, as the domain is torsion-free, $\Lambda_0(G)$ is injective. This implies that the long sequence  \eqref{Gamma(2)} splits into two short exact sequences:

$$0\rightarrow\Z^5 \rightarrow K_0^G(\underline{E}G)\rightarrow \Z\rightarrow 0$$
and
$$0\rightarrow\Z^{10}\rightarrow K_1^G(\underline{E}G)\rightarrow \Z^2\rightarrow 0$$

\begin{Prop}\label{LHSforZ2rtimesGamma(2)}
Let $G=\Z^2\rtimes \Gamma(2)$. The equivariant K-homology groups of $G=\Z^2\rtimes\Gamma(2)$ are
$$
K_0^G(\underline{E}G)=\Z^6 \quad \text{and} \quad K_1^G(\underline{E}G)=\Z^{12}.
$$
Further, the map $\underline{E}\mathbf{p2}\rightarrow \underline{E}G$ induced by inclusion of groups, induces an isomorphism in $K_0^G$.
\end{Prop}

{\bf Proof:} The first statement follows immediately from the two short exact sequences obtained from \eqref{Gamma(2)}. For the second statement, the functoriality of the Atiyah-Hirzebruch spectral sequence, the splitting of \eqref{Gamma(2)} and the fact that there is a 2-dimensional model for $\underline{E}\mathbf{p2}$ (so $H_3^{\mathfrak{F}_{\mathbf{p2}}}(\underline{E}\mathbf{p2},R_{\mathbb{C}})=0$), imply that the inclusion $\mathbf{p2}\hookrightarrow G$ induces a commutative diagram:

$$
\begin{tikzcd}[row sep=scriptsize, column sep=scriptsize]
0 \arrow[r] & H_0^{\mathfrak{F}_G}(\underline{E}G,R_{\mathbb{C}}) \arrow[r] & K_0^G(\underline{E}G) \arrow[r] & H_2^{\mathfrak{F}_G}(\underline{E}G,R_{\mathbb{C}}) \arrow[r] & 0 \\
0 \arrow[r] & H_0^{\mathfrak{F}_{\mathbf{p2}}}(\underline{E}\mathbf{p2},R_{\mathbb{C}}) \arrow[r] \arrow[u] &  K_0^{\mathbf{p2}}(\underline{E}\mathbf{p2}) \arrow[r] \arrow[u] & H_2^{\mathfrak{F}_{\mathbf{p2}}}(\underline{E}\mathbf{p2},R_{\mathbb{C}}) \arrow[r] \arrow[u] & 0  
\end{tikzcd}$$ 

The above analysis of Mart\'{i}nez spectral sequence implies that the vertical homomorphisms on the sides of the diagram are isomorphisms, so by the 5-lemma the same holds for the middle one.\hfill$\square$

\subsection{The right-hand side: K-theory}

Since $C^*_r(G)\simeq C^*(\mathbf{p2})\rtimes S$, we may compute the K-theory of $C^*_r(G)$ using the Pimsner-Voiculescu 6-terms exact sequence for crossed products by a free group (Theorem 3.5 in \cite{PV}). Since $K_1(C^*(\mathbf{p2}))=0$ by Theorem 3.5 in \cite{ELPW}, this sequence unfolds as:
\begin{equation}\label{PVforZ2rtimesGamma(2)}
    0\rightarrow K_1(C^*_r(G))\rightarrow K_0(C^*(\mathbf{p2}))\oplus K_0(C^*(\mathbf{p2}))\stackrel{\beta}{\rightarrow} K_0(C^*(\mathbf{p2}))\rightarrow K_0(C^*_r(G))\rightarrow 0
\end{equation}
where $\beta(x_1,x_2)=x_1-(\alpha_1)_*(x_1)+x_2-(\alpha_2)_*(x_2)$ for $x_1,x_2\in K_0(C^*(\mathbf{p2}))$ and 
$\alpha_1=\begin{pmatrix} 1 & 2 \\0 & 1 \end{pmatrix},
\alpha_2=\begin{pmatrix}
    1 & 0 \\2 & 1
\end{pmatrix}$.

Recall from Subsection \ref{RHS-AS} that $K_0(C^*(\mathbf{p2}))\simeq\Z^6$ has a basis
$$B_2=\Bigl\{[1],\, [\frac{1+\varepsilon}{2}],\,[\frac{1+u\varepsilon}{2}],\,[\frac{1+v\varepsilon}{2}],\,[\frac{1+uv\varepsilon}{2}],\,F_2\Bigr\}.$$
We are led to study the action of $\Gamma(2)$ on $K_0(C^*(\mathbf{p2}))$. That is the content of the next lemma.

\begin{Lem}\label{trivialaction} 
The group $\Gamma(2)$ acts trivially on $K_0(C^*(\mathbf{p2}))$.
\end{Lem}

\begin{proof}
Recall that $\Gamma(2)$ is generated by $\varepsilon, \alpha_1$ and $\alpha_2$. Clearly $\varepsilon$ acts trivially. We prove that $\alpha_1$ also acts trivially, the proof for $\alpha_2$ is similar.

Recall from Subsection \ref{RHS-AS} that $K_0(C^*(\mathbf{p2}))\simeq\Z^6$ has a basis
$$B_2=\Bigl\{[1],\, [\frac{1+\varepsilon}{2}],\,[\frac{1+u\varepsilon}{2}],\,[\frac{1+v\varepsilon}{2}],\,[\frac{1+uv\varepsilon}{2}],\,F_2\Bigr\}.$$ 
Recall that $\alpha_1$ acts on $\mathbf{p2}$ by $\alpha_1(u^kv^m)=u^{k+2m}v^{m}$ and $\alpha_1(\varepsilon)=\varepsilon$. We now tackle the first five basis elements of $B_2$ one by one.
\begin{itemize}
\item It is clear that $(\alpha_1)_*[1]=[1]$.
\item Similarly $(\alpha_1)_*[\frac{1+\varepsilon}{2}]=[\frac{1+\varepsilon}{2}]$.
\item Similarly $(\alpha_1)_*[\frac{1+u\varepsilon}{2}]=[\frac{1+u\varepsilon}{2}]$.
\item We have $(\alpha_1)_*[\frac{1+v\varepsilon}{2}]=[\frac{1+u^2v\varepsilon}{2}]$ but $u^{-1}(\frac{1+u^2v\varepsilon}{2})u=\frac{1+v\varepsilon}{2}$.
\item We have $(\alpha_1)_*[\frac{1+uv\varepsilon}{2}]=[\frac{1+u^3v\varepsilon}{2}]$ but $u^{-1}(\frac{1+u^3v\varepsilon}{2})u=\frac{1+uv\varepsilon}{2}$.
\end{itemize}

To prove that $(\alpha_1)_*(F_2)=F_2$, we appeal to the explicit description of the Baum-Connes assembly map for certain wallpaper groups due to Davis and L\"uck \cite{DaLu13}: for a group $\Gamma=\Z^n \rtimes C_p$, with $p$ prime and $C_p$ acting freely on $\Z^n\setminus\{0\}$, there is  a natural short exact sequence (Theorem 8.1.(iii) in \cite{DaLu13}):
$$0\rightarrow\bigoplus_{(P)\in \mathcal{P}}\Tilde{R}_\mathbb C(P)\rightarrow K_0^\Gamma(\underline{E}\Gamma)\rightarrow K_0(\underline{B}\Gamma)\rightarrow 0$$
where $\mathcal{P}$ denotes the set of conjugacy classes of finite subgroups in $\Gamma$, and $\Tilde{R}_\mathbb C(P)$ denotes the kernel of the dimension map $R_\mathbb C(P)\rightarrow \Z:[V]\mapsto \dim_\mathbb C(V)$. This sequence splits naturally \footnote{Naturality, which is important for our purpose, follows from Theorem 0.7 in \cite{Lueck}; see also the last line of p. 411 in \cite{DaLu13}.} after tensoring with $\Z[1/p]$.

We apply this to $\Gamma=\mathbf{p2}=\Z^2\rtimes C_2$. We observe that:
\begin{itemize}
    \item $\alpha_1$ acts trivially on $\bigoplus_{(P)\in \mathcal{P}}\Tilde{R}_\mathbb C(P)$: indeed the image of this subgroup by $\mu_0^\mathbf{p2}$ is contained in the subgroup of $K_0(C^*(\mathbf{p2}))$ spanned by $B_2\setminus\{F_2\}$, where $\alpha_1$ is known to act trivially by the first part of the proof.
    \item As $\underline{B}\mathbf{p2}$ is homeomorphic to the 2-sphere $S^2$, and $\alpha_1$ preserves the orientation of $S^2$, the action of $\alpha_1$ on $K_0(\underline{B}\mathbf{p2})$ is trivial.
\end{itemize}

Since the sequence 
$$0\rightarrow (\bigoplus_{(P)\in \mathcal{P}}\Tilde{R}_\mathbb C(P))\otimes \Z[1/2] \rightarrow K_0^\mathbf{p2}(\underline{E}\mathbf{p2})\otimes\Z[1/2] \rightarrow K_0(\underline{B}\mathbf{p2})\otimes\Z[1/2]\rightarrow 0$$
splits naturally, $\alpha_1$ acts trivially on 
$K_0^\mathbf{p2}(\underline{E}\mathbf{p2})\otimes\Z[1/2]$. 
Since $K_0^\mathbf{p2}(\underline{E}\mathbf{p2})$ is torsion-free, the canonical map $K_0^\mathbf{p2}(\underline{E}\mathbf{p2})\rightarrow K_0^\mathbf{p2}(\underline{E}\mathbf{p2})\otimes\Z[1/2]$ is injective, so $\alpha_1$ acts trivially on $K_0^\mathbf{p2}(\underline{E}\mathbf{p2})$, hence also on $K_0(C^*(\mathbf{p2}))$. 
Of course the case of $\alpha_2$ is completely analogous. This concludes the proof.
\end{proof} 

\begin{Thm}\label{RHSforZ2rtimesGamma(2)}
Let $G=\Z^2 \rtimes \Gamma(2)$.  We have that 
$$K_0(C^*_r(G))=\Z^6\quad \text{and} \quad K_1(C^*_r(G))=\Z^{12}.$$
 Further, the assembly map $\mu_0^G: K_0^G(\underline{E}G)\rightarrow K_0(C^*_r(G))$ is an isomorphism.
\end{Thm}

{\bf Proof:} 
    To show the first part, note that by Lemma \ref{trivialaction}, we have that $\beta=0$. Plug this in the 
    exact sequence \eqref{PVforZ2rtimesGamma(2)}, the sequence splits as 
    $$0\rightarrow K_0(C^*(\mathbf{p2}))\rightarrow K_0(C^*_r(G))\rightarrow 0$$
    and 
    $$0\rightarrow K_1(C^*_r(G))\rightarrow K_0(C^*(\mathbf{p2}))\oplus K_0(C^*(\mathbf{p2}))\rightarrow 0,$$
    so that $K_0(C^*(\mathbf{p2}))=\Z^6$ is isomorphic to $ K_0(C^*_r(G))$, while $ K_1(C^*_r(G))$ is isomorphic to the direct sum of two copies of $K_0(C^*(\mathbf{p2}))=\Z^6$.
For the second statement, note that by the first part of the proof and by the second part of Proposition \ref{LHSforZ2rtimesGamma(2)}, the inclusion $\iota: \mathbf{p2}\rightarrow G$ induces isomorphisms at the level of $K_0$, on both sides of the conjecture. By naturality of the assembly map we have
    $$\mu_0^G\circ\iota_*=\iota_*\circ \mu_0^\mathbf{p2}.$$  Since $\mu_0^\mathbf{p2}$ is an isomorphism, so is $\mu_0^G$.\hfill$\square$
    



Authors addresses:\\

\medskip
\noindent
Departamento de Matem\'aticas\\
Universidad de Sevilla\\
c/ Tarfia, s/n \\
41018 Sevilla - SPAIN\\
ramonjflores@us.es

\medskip
\noindent
Institute of Mathematics\\
University of Potsdam\\
Campus Golm, Haus 9\\
Karl-Liebknecht-Str. 24-25\\
D-14476 Potsdam OT Golm - GERMANY\\
sanaz.pooya@uni-potsdam.de

\medskip
\noindent
Institut de math\'ematiques\\
Universit\'e de Neuch\^atel\\
11 Rue Emile Argand - Unimail\\
CH-2000 Neuch\^atel - SUISSE\\
alain.valette@unine.ch


\begin{thebibliography}{CCJJV}

\bibitem[BCH94]{BCH}
P. {\sc Baum}, A. {\sc Connes}, and N. {\sc Higson},
\newblock{\em Classifying space for proper actions and K-theory of group $C^*$-algebras},
\newblock Contemp. Math., 167, pp. 240--291, Amer. Math. Soc., Providence, RI, 1994.

\bibitem[BT16]{BuTa} M. {\sc Bucher} and A. {\sc Talambutsa},
\newblock {\em Exponential growth rates of free and amalgamated products},
\newblock  Israel J. Math. 212 (2016), 521-546.

\bibitem[Bur91]{Burger} M. {\sc Burger},
\newblock {\em Kazhdan constants for $SL(3,\mathbb{Z})$},
\newblock  J. Reine Angew. Math. 413 (1991), 36-67.

\bibitem[Con]{Conrad} K. {\sc Conrad},
\newblock $SL_2(\Z)$,
\newblock Undated notes available on \verb"https://kconrad.math.uconn.edu/blurbs/grouptheory/SL(2,Z).pdf"

\bibitem[DaLu98]{DaLu98}
\newblock J. {\sc Davis} and W. {\sc L\"{u}ck}.
\newblock{\em Spaces over a category and assembly maps in isomorphism conjectures in K- and L-Theory},
\newblock K-theory 15, (1998), 201 -- 252.

\bibitem[DaLu13]{DaLu13} J. {\sc Davis} and W. {\sc L\"{u}ck}.
\newblock{\em The topological K-theory of certain crystallographic groups},
\newblock J. Noncommut. Geom. 7 (2013), 373-431.

\bibitem[ELPW10]{ELPW} S. {\sc Echterhoff}, W. {\sc L\"uck}, N. C. {\sc Phillips}, and S. {\sc Walters},
\newblock {\em The structure of crossed products of irrational rotation algebras by finite subgroups of $SL_2(\Z)$},
\newblock J. Reine Angew. Math., 639 (2010), 173-221.

\bibitem[Flo21]{Flores} R. {\sc Flores},
\newblock {\em Bredon homology of wallpaper groups},
\newblock \verb"arXiv:2111.14245", to appear in J. Korean Math. Soc.

\bibitem[Flu11]{Flu11} M. {\sc Fluch},
\newblock {\em On Bredon (co-)homological dimensions of groups},
\newblock PhD Thesis, \verb"arXiv:1009.46339".

\bibitem[FPV17]{FPV17} R. {\sc Flores}, S. {\sc Pooya} and A. {\sc Valette},
\newblock{\em K-homology and K-theory for the lamplighter groups of finite groups},
\newblock Proc. Lond. Math. Soc., 115 (2017), 1207-1226.


\bibitem [GJV19]{GoJuVa} M.P. {\sc Gomez Aparicio}, P. {\sc Julg} and A. {\sc Valette},
\newblock {\em The Baum-Connes conjecture: an extended survey},
\newblock In. A. Chamseddine et al. (eds.), {\it Advances in Noncommutative Geometry}, Springer Nature AG Switzerland 2019, 127-244.

\bibitem[Hat02]{Hatcher} A. {\sc Hatcher},
\newblock{\em Algebraic Topology},
\newblock Cambridge University Press, 2002.

\bibitem[HK01]{HiKa} N. {\sc Higson} and G.G. {\sc Kasparov},
\newblock {\em E-theory and KK-theory for groups which act properly and isometrically on Hilbert space},
\newblock Invent. Math. 144 (2001), 23-74.

\bibitem[Is11]{Isely} O. {\sc Isely},
\newblock {\em K-theory and K-homology for semi-direct products of $\Z^2$ by $\Z$},
\newblock PhD thesis, Universit\'e de Neuch\^atel, 2011.

\bibitem[JV 84]{JuVa} P. {\sc Julg} and A. {\sc Valette},
\newblock {\em K-theoretic amenability for $SL_2(\mathbb{Q}_p)$ and the action on the associated tree}, 
\newblock J. Funct. Anal. 58 (1984), 194-215.

\bibitem[Laf02]{Laff} V. {\sc Lafforgue},
\newblock {\em K-th\'eorie bivariante pour les alg\`ebres de Banach et conjecture de Baum-Connes},
\newblock Invent. Math. 149 (2002), 1-95.

\bibitem[Lu00]{Lueck00} W. {\sc L\"uck},
\newblock{\em The type of the classifying space for a family of subgroups
},
\newblock J. Pure Appl. Algebra (2000), 177-203.

\bibitem[Lu02]{Lueck} W. {\sc L\"uck},
\newblock{\em The relation between the Baum-Connes Conjecture and the Trace Conjecture
},
\newblock Invent. Math. 149 (2002), 123-152.

\bibitem[Lu05]{Lueck05}
W. {\sc L\"uck},
\newblock{\em Survey on classifying spaces for families of subgroups },
\newblock Prog. Math. 248 (2005), 269-322.

\bibitem[Lu23]{Lu23}  W. {\sc L\"uck},
{\it Isomorphism conjectures in $K$ and $L$-theory},
in preparation, https://www.him.uni-bonn.de/lueck/data/ic.pdf

\bibitem[Mar02]{Martinez} C. {\sc Mart\'{i}nez-P\'{e}rez},
\newblock {\em  A spectral sequence in Bredon (co)homology}
\newblock J. Pure Appl. Algebra 176 (2002), no. 2-3, 161-173.


\bibitem[MV03]{MV03} G. {\sc Mislin} and A. {\sc Valette},
\newblock{\em Proper group actions and the Baum-Connes conjecture}
\newblock Ads. Courses in Math. CRM Barcelona, Birkh\"auser, 2003.


\bibitem[Na85]{Natsume} T. {\sc Natsume},
\newblock {\em On $K_*(C^*(SL_2(\Z)))$ (Appendix to ``K-theory for certain group $C^*$-algebras'' by E.C. Lance)},
\newblock J. Operator Theory 13 (1985), 103-118.

\bibitem[OO01]{Oyono} H. {\sc Oyono-Oyono},
\newblock {\em Baum-Connes conjecture and group actions on trees},
\newblock K-Theory 24 (2001), 115-134.

\bibitem[Pi86]{Pimsner} M. {\sc Pimsner},
\newblock {\em KK-groups of crossed products by groups acting on trees},
\newblock Invent. Math., 86 (1986), 603-634.

\bibitem[PV82]{PV} M. {\sc Pimsner} and D. {\sc Voiculescu},
\newblock {\em K-groups of reduced crossed products by free groups},
\newblock J. Operator Theory, 8 (1982), 131-156.

\bibitem[PV80]{PV80} M. {\sc Pimsner} and D. {\sc Voiculescu},
\newblock {\em Exact sequences for K-groups and Ext-groups of certain cross-product C*-algebras},
\newblock J. Operator Theory, 4 (1980), 93-118.


\bibitem[Rob82]{Rob82} D. {\sc Robinson},
\newblock {\em A course in the theory of groups},
\newblock Rad. Texts in Math., 80, Springer-Verlag, New York-Berlin, 1982.

\bibitem[Sch78]{Sch78} D. {\sc Schattschneider},
\newblock {\em The plane symmetry groups: their recognition and notation},
\newblock Amer. Math. Monthly, 85 (1978), 439--450.

\bibitem[Se77]{Serre} J.-P. {\sc Serre},
\newblock {\em Arbres, amalgames, $SL_2$},
\newblock Ast\'erisque 46, Soc. Math. France, 1977.



\bibitem[Wa18]{Walters} S. {\sc Walters},
\newblock {\em K-theory of rotation algebra crossed products by amalgamated products of finite cyclic groups},
\newblock \verb"arXiv:1809.09286".

\bibitem[WO93]{WO} N.E. {\sc Wegge-Olsen},
\newblock{\em K-theory and $C^*$-algebras,}
\newblock Oxford Univ. Press, 1993.

\bibitem[Wil07]{DPW} D.P. {\sc Williams},
\newblock{\em Crossed products of $C^*$-algebras},
\newblock Math. Surveys \& Monographs 134, Amer. Math. Soc. 2007.

\bibitem[Wik23]{Wiki} {\sc Wikipedia},
\newblock{\em Ping-pong lemma},
\newblock \verb|https://en.wikipedia.org/wiki/Ping-pong_lemma|; retrieved 2023-09-26.

\bibitem[Ya97]{Yang} M. {\sc Yang}
\newblock {\em Crossed products by finite groups acting on low dimensional complexes and applications},
\newblock PhD thesis, Univ. Saskatchewan, Saskatoon, 1997.




\end{thebibliography}
\end{document}